\documentclass[11pt]{amsart}
\usepackage{amsmath,amssymb,mathrsfs, xcolor}
\usepackage{parskip} 
\usepackage{mathtools}
\usepackage{hyperref}
\hypersetup{linkbordercolor = {white}, citebordercolor = {white},}
\usepackage[alphabetic]{amsrefs}

\usepackage[alphabetic]{amsrefs}
\numberwithin{equation}{section}

\theoremstyle{definition}
\newtheorem{thm}[equation]{Theorem} 
\theoremstyle{definition}
\newtheorem{definition}[equation]{Definition}
\newtheorem{prop}[equation]{Proposition} 
\newtheorem{cor}[equation]{Corollary}

\newtheorem{lemma}[equation]{Lemma}
\newtheorem{remark}[equation]{Remark}

\newcommand{\ovl}[1]{\overline{#1}}
\newcommand{\udl}[1]{\underline{#1}}
\newcommand{\pair}[1]{\langle #1\rangle}
\newcommand{\set}[1]{\left\{#1\right\}}
\newcommand{\ppair}[1]{\left\langle #1\right\rangle}
\newcommand{\pr}[1]{\left(#1\right)}
\newcommand{\abs}[1]{\left|#1\right|}
\newcommand{\bb}[1]{\mathbb{#1}} \newcommand{\td}[1]{\widetilde{#1}}

\newcommand{\tr}{\text{tr}}
\newcommand{\D}{\Delta}
\newcommand{\Hess}{\text{Hess}}
\newcommand{\n}{\nabla}
\newcommand{\sbst}{\subseteq}

\newcommand{\bd}{\partial}

\title{Almost isoclinic Lagrangian submanifolds
}

\author{Tang-Kai Lee}
\author{Mu-Tao Wang}

\address{Department of Mathematics, Columbia University, New York, NY 10027, USA}
\email{leetk@math.columbia.edu and mtwang@math.columbia.edu}

\date{\today}

\begin{document}
\begin{abstract}
We introduce the almost isoclinic region in the oriented Lagrangian Grassmannian ${\rm Lag}^+(n)$ of $\mathbb{C}^n$, an intrinsic higher-dimensional analog of a natural convex region in ${\rm Lag}^+(2) \simeq \mathbb S^1\times \mathbb S^2$. 
A Lagrangian submanifold is called almost isoclinic if its Gauss map takes values in this region, extending the graphical condition that the characteristic angles of the tangent plane remain uniformly close. 
We construct a canonical positive function $\Lambda$ on this region and prove that $\log \Lambda $ is concave with respect to the invariant Grassmannian metric. 
This property yields subharmonicity and monotonicity formulas for minimal Lagrangians and Lagrangian mean curvature flow. 
As applications, we prove rigidity and Bernstein-type results, including that a complete connected almost isoclinic minimal Lagrangian with a positive lower bound for $\Lambda$ must be a Lagrangian $n$-plane.
\end{abstract}

\maketitle

\section{\bf Introduction}

Lagrangian submanifolds in $\mathbb{C}^n$ sit at the intersection of symplectic geometry, calibrated geometry, and geometric analysis. Their tangent spaces define a Gauss map into the Lagrangian Grassmannian, and geometric equations for such submanifolds imposed conditions on this map. 
In particular, special Lagrangian geometry and Lagrangian mean curvature flow motivate the search for distinguished domains in the Lagrangian Grassmannian that encode useful geometric and analytic structure. 
In dimension two, the picture is especially concrete: the oriented Lagrangian Grassmannian is isomorphic to $\mathbb S^1\times \mathbb S^2$ and within this model a natural convex region provides the motivation  for the intrinsic higher-dimensional analog sought here.

We explain the terminology \textit{almost isoclinic} in the graphical setting first. 
When $n=2$, if a Lagrangian surface in $\mathbb{C}^2$ is given locally as the graph of $\nabla u$, and if the eigenvalues of $\nabla^2 u$ are $\tan\theta_1$ and $\tan\theta_2,
\,-\frac{\pi}{2}<\theta_j<\frac{\pi}{2}, \,j=1,2$, then the almost isoclinic condition becomes 
\[ |\theta_1-\theta_2|<\frac{\pi}{2}. \] 
More generally, for a Lagrangian graph in $\mathbb{C}^n$ with Hessian eigenvalues $\tan\theta_1,\dots,\tan\theta_n,\, -\frac{\pi}{2}<\theta_j<\frac{\pi}{2}, j=1,\cdots n$, the inequalities 
\begin{equation} \label{eq:AI_angle}
|\theta_j-\theta_k|<\frac{\pi}{2}
\qquad (1\le j<k\le n) 
\end{equation} 
express that the characteristic angles of the tangent plane remain uniformly close to one another. 
This explains the term almost isoclinic. 
The above formulation is inherently tied to the graphical setting: it depends on a choice of charts and excludes natural limiting configurations in which some of the angles $\theta_j$ reach the boundary values $\pm \pi/2$. 
A basic objective of this paper is therefore to replace the graphical angle inequalities by an invariant condition formulated directly on the oriented Lagrangian Grassmannian. 

Motivated by the two-dimensional model, we define a distinguished domain in the oriented Lagrangian Grassmannian of $\mathbb{C}^n$ for arbitrary $n$, which we call the \textit{almost isoclinic region}, denoted by $\mathcal{AI}(n).$ 
This region is intrinsically characterized in terms of the invariant geometry of the Lagrangian Grassmannian. 
We say that a Lagrangian submanifold is almost isoclinic if its Gauss map takes values entirely in the almost isoclinic region. 
In this formulation, the condition is coordinate-free and global, yet flexible enough to extend far beyond the class of graphical Lagrangians.  
We provide examples of almost isoclinic Lagrangian submanifolds that show genuinely non-graphical behavior; see Figure~\ref{fig:spiral} in Section ~\ref{sym-exam} .

A central feature of the almost isoclinic region is the existence of a canonically defined function~$\Lambda$. 
For a Lagrangian graph in $\mathbb{C}^n$, the pull-back of $\Lambda $ via the Gauss map is given by 
\begin{equation}
\prod_{j<k} \cos(\theta_j-\theta_k).
\end{equation}
The definition of $\Lambda$ can be extended beyond the graph chart and remains manifestly positive on the almost isoclinic region.

Our main structural result is the log-concavity property of $\Lambda$ on this region with respect to the invariant Grassmannian metric.

\begin{thm}
\label{thm:main-thm-log-concavity}
    The function
    $\Lambda$ is log-concave on the almost isoclinic region $\mathcal{AI}(n)$ with respect to the Grassmannian metric on ${\rm Lag}^+(n)$.
    Equivalently, 
    $\Hess\pr{\log \Lambda(S)}(A, A)
    \le 0$
    for every $S\in \mathcal{AI}(n)$ and every tangent vector $A\in T_S{\rm Lag}^+(n)$.
\end{thm}

Theorem~\ref{thm:main-thm-log-concavity} is the fundamental analytic ingredient underlying the applications developed in this paper.
We show that the almost isoclinic condition imposes subharmonicity or monotonicity properties for $\Lambda$ on minimal Lagrangians and along Lagrangian mean curvature flows, leading to rigidity and regularity consequences. 
As an application, we prove that any almost isoclinic minimal Lagrangian regular cone must be a union of Lagrangian $n$-planes, which implies the following Bernstein-type result.

\begin{thm}
	\label{thm:Bernstein-AC}
	Let $M$ be a complete connected almost isoclinic minimal Lagrangian in $\mathbb{C}^n$.
    If there exists $c>0$ such that $\Lambda\ge c,$ then $M$ is a Lagrangian $n$-plane.
\end{thm}

The Bernstein theorem is known to have an equivalent form of a curvature estimate which is recorded in Corollary~\ref{cor:local-curv-est}.
The positivity of $c$ in the bound is necessary especially when $n$ is odd as the $n$-dimensional Lagrangian catenoid satisfies $\Lambda\ge 0.$
When the dimension of $M$ is two, we obtain a stronger rigidity result based on complex analysis arguments; see Theorem~\ref{thm:Bernstein} and Remark~\ref{rmk:sharpness}.


An important special case of Theorem~\ref{thm:Bernstein-AC} occurs when $M$ is known to be an entire graph with bounded gradient.
Bernstein-type theorems for entire special Lagrangian graphs have been established in various situations;
see \cite{Y02,TW02,Y06}.
In particular, in a form relevant to the present work (cf. Lemma~\ref{lem:AI-graph}), if \(u:\mathbb{R}^n\to\mathbb{R}\) is an entire solution of the special Lagrangian equation and \(\lambda_j\)'s are the eigenvalues of \(\n^2 u\), then \(u\) is a quadratic polynomial provided that $\n^2 u$ is bounded and $\lambda_j\lambda_k\geq -3/2$.\footnote{In comparison, the almost isoclinic condition is $\lambda_j\lambda_k>-1$ but an almost isoclinic submanifold need not be graphical.}

The theory identifies a geometric regime that balances structure and flexibility, which is restrictive enough for convexity methods, yet broad enough to capture global and singularity-forming phenomena. 
It also singles out a natural domain in the Lagrangian Grassmannian where symmetric-space geometry can interact effectively with minimal submanifold theory and Lagrangian mean curvature flow.
This paper develops the foundational properties of the almost isoclinic region and establishes the principal geometric and analytic consequences for minimal Lagrangian geometry.
Existence results and regularity properties of almost isoclinic submanifolds along a Lagrangian mean curvature flow will be investigated in future work.

We organize the paper as follows.
In Section~\ref{sec:Pre}, we introduce the geometry of the oriented Lagrangian Grassmannian ${\rm Lag}^+(n),$ define the almost isoclinic region and almost isoclinic submanifolds, and construct the canonical function $\Lambda$.
In Section~\ref{sec:log-concavity}, we recall the Grassmannian metric on ${\rm Lag}^+(n)$ and discuss the log-concavity property of $\Lambda$ and its consequences.
In Section~\ref{sec:ex}, we present examples of almost isoclinic submanifolds.
In Section~\ref{sec:proof-Lambda}, we evaluate the Hessian of $\Lambda$ which implies Theorem~\ref{thm:main-thm-log-concavity}.
Finally, in Section~\ref{sec:Bernstein}, we prove Theorem~\ref{thm:Bernstein-AC} and related Bernstein-type results.

\subsection*{\bf Acknowledgment}
Part of the work was completed while T.-K. Lee and M.-T. Wang were visiting Academia Sinica, the National Center of Theoretical Sciences, and National Taiwan University, and they thank all three institutions for their hospitality. 
T.-K. Lee acknowledges the support from Ming-Lun Hsieh, Man-Chun Lee, Connor Mooney, Chung-Jun Tsai, and NSF grant DMS-2533558.
M.-T.~Wang was partially supported by the National Science Foundation under Grant No. DMS-2404945, and by a Simons Foundation Travel Grant.

\section{\bf Preliminaries}
\label{sec:Pre}

\subsection{Almost isoclinic Lagrangian surfaces in $\mathbb{C}^2$}
\label{sec:2D-AI}

Let \((z_1,z_2)\) denote the standard complex coordinates on \(\mathbb{C}^2\), and let
\(
\frac{i}{2}\left(dz_1\wedge d\bar z_1+dz_2\wedge d\bar z_2\right)
\)
be the standard Kähler form. 
Recall that an oriented surface is Lagrangian if the restriction of the Kähler form to the surface vanishes. 
For such a surface, we define \begin{equation}\label{Lambda_n=2}
\Lambda=*\operatorname{Re}\left(dz_1\wedge d\bar z_2\right),
\end{equation} 
where $*$ is the Hodge star on the surface.  
In this setting, a Lagrangian surface is called \textit{almost isoclinic} if $\Lambda>0,$ i.e., the restriction of $\operatorname{Re}\left(dz_1\wedge d\bar z_2\right)$ to the surface is positive. 
In contrast, the almost calibrated condition
(see \cite[Remark~5.1]{W01-JDG}) corresponds to the positivity of the
restriction of
\(
\operatorname{Re}\left(dz_1\wedge dz_2\right).
\)

As noted in the introduction, if the Lagrangian surface is given as the graph
of the gradient \(\nabla u\) of a function \(u\) defined on \(\mathbb{R}^2\),
then the ``slope'' of the surface is encoded by the Hessian \(\nabla^2 u\).
In this setting, the almost isoclinic condition becomes
\begin{align}\label{2d-theta}
|\theta_1-\theta_2|<\frac{\pi}{2},
\end{align}
where the eigenvalues of \(\nabla^2 u\) are \(\tan\theta_1\) and
\(\tan\theta_2\).

Another way to describe this condition is through the oriented Lagrangian Grassmannian \(\operatorname{Lag}^+(2)\). 
Using the decomposition of two-forms into self-dual and anti-self-dual parts, the space of oriented two-planes in \(\mathbb{R}^4\) can be identified with \(\mathbb S^2\times \mathbb S^2\), where the first factor records the self-dual component and the second records the anti-self-dual component. 
The Kähler form is self-dual, so the Lagrangian condition cuts out an equator in the self-dual factor; thus \(\operatorname{Lag}^+(2)\) can be identified with \(\mathbb S^1\times \mathbb S^2\). 
On the other hand, the two-form $\operatorname{Re}\left(dz_1\wedge d\bar z_2\right)$ is anti-self-dual. 
Hence, the almost isoclinic condition requires the anti-self-dual component of the Gauss map to lie in an upper hemisphere of the anti-self-dual two-sphere. 
Equivalently, the Gauss map of an almost isoclinic surface takes values in the product of \(\mathbb S^1\) with this upper hemisphere \(\mathbb S^2_+\), a manifestly convex subset of \(\mathbb S^1\times \mathbb S^2\) on which $\Lambda$ is positive.

\begin{figure}[h]
	\centering
	\includegraphics[width=7cm]{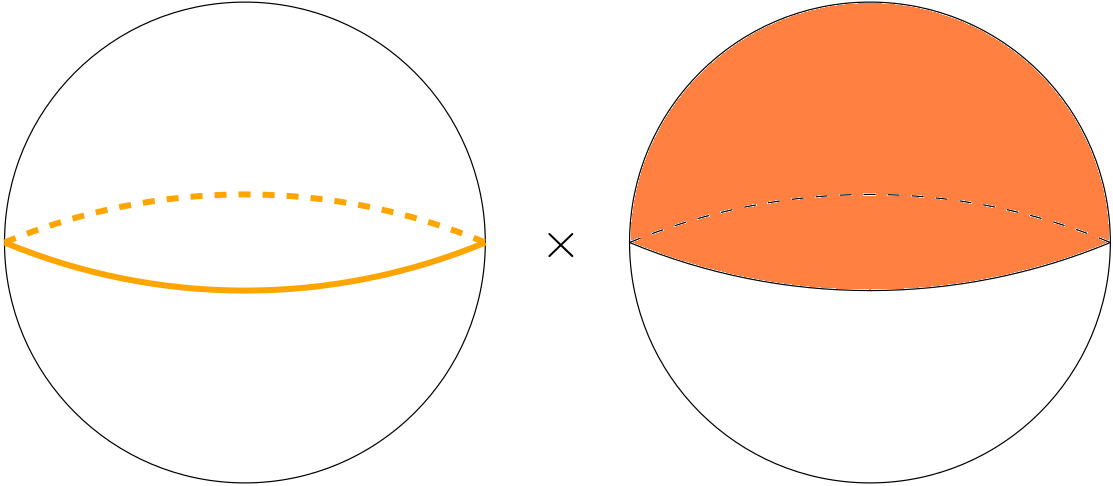}
	\caption{The oriented Grassmannian of two-planes in $\bb R^4$ can be identified with~$\mathbb S^2\times \mathbb S^2,$ and the almost isoclinic region can be identified with $\mathbb S^1\times \mathbb S^2_+,$ the orange region, whose area is exactly half of that of ${\rm Lag}^+(2)\simeq \mathbb S^1\times \mathbb S^2.$}
    \label{fig:Gauss}
\end{figure}

In \cite{W01-MRL}, the graph $\Sigma$ of a symplectomorphism from $(\mathbb{C},\omega_1=dx_1\wedge dy_1)$ to $(\mathbb{C},\omega_2=dx_2\wedge dy_2)$ is regarded as an embedded surface in $\mathbb{C}^2$. With respect to the symplectic form $\omega_1-\omega_2$ on $\mathbb{C}^2$, $\Sigma$ is Lagrangian and the restriction of $\omega_1+\omega_2$ to $\Sigma$ is positive. 
After choosing the orientation for which $\omega_1-\omega_2$ is self-dual and $\omega_1+\omega_2$ is anti-self-dual, this condition is seen to be equivalent to the almost isoclinic condition for such a graphical surface.

However, this picture is special to the case \(n=2\). 
\(\operatorname{Lag}^+(n)\) for $n\geq 3$ no longer decomposes as a product analogous to \(\mathbb S^1\times \mathbb S^2\); apart from the Maslov phase direction, its geometry is irreducible. 
In general dimensions, we must therefore work in a graph chart of \(\operatorname{Lag}^+(n)\) to define the almost isoclinic condition as in \eqref{eq:AI_angle}, and then extend this condition beyond the graphical
setting.

\subsection{Description of the oriented Lagrangian Grassmannian ${\rm Lag}^+(n)$}
In $\bb C^n \simeq \bb R^{2n}$ with coordinates $z_j=x_j+iy_j,j=1,\cdots,n,$ consider the symplectic form $\omega = dx_1\wedge dy_1 + \cdots + dx_n\wedge dy_n$
and the orientation given by $dx_1\wedge dy_1\wedge\cdots\wedge dx_n\wedge dy_n.$
An $n$-plane $L\sbst\bb R^{2n}$ is called a \textit{ Lagrangian $n$-plane} if $\omega|_L=0.$
Let ${\rm Lag}^+(n)$ be the \textit{oriented Lagrangian Grassmannian}, i.e., the space of oriented Lagrangian linear subspaces in $\bb R^{2n}.$

Any $L \in {\rm Lag}^{+}(n)$ admits an oriented orthonormal basis in \(\mathbb{R}^{2n}\). 
Equivalently, such an $L$ can be represented by a
\(2n \times n\) matrix of the form
$\begin{pmatrix}
X \\
Y
\end{pmatrix}$
whose columns form an oriented orthonormal basis of $L$. Here \(X\) and \(Y\) are real \(n \times n\) matrices and the orthonormality of the basis gives
\begin{equation}\label{ortho}
X^{T}X+Y^{T}Y=I,
\end{equation}
while the Lagrangian condition gives
\begin{equation}\label{Lag}
X^{T}Y-Y^{T}X=0.
\end{equation}
Together, these conditions are equivalent to $X+iY \in U(n).$
We will freely use either
$\begin{pmatrix}
X \\
Y
\end{pmatrix}$ 
or $X+iY$ to represent the same oriented Lagrangian $n$-plane \(L\).

We let $L_B$ denote the reference base $n$-plane, i.e., the $n$-plane represented by $\begin{pmatrix}
I\\0
\end{pmatrix}.$
Note that an $n$-plane $L=\begin{pmatrix}
X\\Y
\end{pmatrix}\in{\rm Lag}^+(n)$ is graphical over $L_B$ if and only if $X$ is invertible.  
In this case, $L$ is the graph of the matrix $YX^{-1},$ which is called the graph matrix of $L$ over $L_B.$ 
Conditions~\eqref{ortho} and~\eqref{Lag} imply that $YX^{-1}$ is symmetric. 
We first collect general identities for such $X$ and $Y.$

\begin{lemma}
	\label{lem:detX>0}
	Let $n\ge 2.$ 
	Suppose $X$ and $Y$ are $n\times n$ matrices that satisfy \eqref{ortho} and \eqref{Lag}. 
	If $X$ is invertible, then
	\begin{align*}
	(\det X)^2
	= \prod_{k=1}^n \frac 1{1+\lambda_k^2}, 
	\end{align*} where $\lambda_k$'s are the eigenvalues of the matrix $YX^{-1}$.
\end{lemma}

\begin{proof}
	The conditions \eqref{ortho} and \eqref{Lag} imply that $YX^{-1}$ is symmetric, and we diagonalize  $YX^{-1}$ by
	\begin{align*}
	Q^{T}\pr{YX^{-1}}Q 
	=D
	:= {\rm diag}\pr{\lambda_1,\cdots,\lambda_n}
	\end{align*}
	for some $Q\in {SO}(n)$ and $\lambda_k\in\bb R.$
	This implies $Q^{T}YQ = DQ^{T}XQ,$ which, along with \eqref{ortho}, implies
	\begin{align*}
	I= Q^TX^TXQ + Q^TY^TYQ
	& = Q^TX^TXQ + Q^TX^TQ D^T\cdot 
	DQ^{T}XQ
	= Q^TX^TQ \pr{I + D^2} Q^TXQ,
	\end{align*}
	so 
	$Q^TX^{-T}X^{-1}Q
	= I+D^2
	= {\rm diag}\pr{1+\lambda_1^2,\cdots,1+\lambda_n^2}.$
	Since $\det Q=1,$ this implies the desired formula.
\end{proof}

Next, we provide transformation formulas for eigenvalues of the graph matrices when we rotate an $n$-plane by a ${U}(1)$-action.
Later, we will use these formulas to study the rotational invariance of a distinguished region in ${\rm Lag}^+(n).$

\begin{lemma}
	\label{lem:lambda-rotate}
	Suppose that $X$ and $Y$ are $n\times n$ matrices satisfying \eqref{ortho} and \eqref{Lag}. 
	Define real matrices $\widetilde X$ and $\widetilde Y$ by
	\[
	\widetilde X+i\widetilde Y=e^{-i\varphi}(X+iY).
	\]
	Then $\widetilde X$ and $\widetilde Y$ also satisfy \eqref{ortho} and \eqref{Lag}. 
	Moreover, if both $X$ and $\widetilde X$ are invertible, then
	the eigenvalues $\widetilde\lambda_k$'s of
	$\widetilde Y\widetilde X^{-1}$ and the eigenvalues $\lambda_k$'s of
	$YX^{-1}$ are related by
	\begin{align*}
	\widetilde\lambda_k
	=
	\frac{\lambda_k\cos\varphi-\sin\varphi}
	{\cos\varphi+\lambda_k\sin\varphi}
	\end{align*}
	and there is a transformation formula
	\begin{align}\label{product-transf}
	1 + \td\lambda_j \td\lambda_k
	= \frac{1+\lambda_j\lambda_k}{\pr{\cos\varphi + \sin\varphi\cdot \lambda_j}\pr{\cos\varphi + \sin\varphi\cdot \lambda_k}}
	\end{align}
	for any $j$ and $k.$
\end{lemma}

\begin{proof}
	Since both $X+iY$ and $e^{-i\varphi}I$ are in ${U}(n),$ it follows that $\td X + i\td Y\in {U}(n).$
	As a result, $\td X $ and $\td Y$ satisfy \eqref{ortho} and \eqref{Lag}.
	By $e^{-i\varphi}=\cos\varphi - i\sin\varphi,$ a direct multiplication leads to 
	\begin{align*}
	\td X
	&= \cos\varphi\cdot X 
	+ \sin\varphi\cdot Y \,
	\text{ and }\,
	\td Y
	= -\sin\varphi\cdot X 
	+ \cos\varphi\cdot Y.
	\end{align*}
	Thus, 
	\begin{align*}
	\td Y \td X^{-1}
	& = \pr{-\sin\varphi\cdot X 
		+ \cos\varphi\cdot Y}
	\pr{\cos\varphi\cdot X 
		+ \sin\varphi\cdot Y }^{-1}\\
	& = \pr{-\sin\varphi+ \cos\varphi \cdot Y X^{-1}} 
	\pr{\cos\varphi + \sin\varphi\cdot Y X^{-1}}^{-1} ,
	\end{align*}
	so the eigenvalues of $\td Y \td X^{-1}$ are
	\begin{align*}
	\td \lambda_{k}
	= \frac{-\sin\varphi + \cos\varphi\cdot \lambda_{k}}{\cos\varphi + \sin\varphi\cdot \lambda_k},
	\end{align*}
	and \eqref{product-transf} follows for any $j$ and $k.$
\end{proof}

We will study rotational invariance of subsets in ${\rm Lag}^+(n).$
Thus, for $\varphi\in\bb R,$ the graph chart is defined to be
\begin{align*}
\mathcal O_\varphi
:= \set{
	L\in {\rm Lag}^+(n)
	:
	L \text{ is graphical over } e^{i\varphi}L_B
}.
\end{align*}
Each $\mathcal O_\varphi$ can be identified with the space of symmetric $n\times n$ matrices. 

\begin{lemma}
	Suppose $L\in {\rm Lag}^+(n)$ can be represented by $X+iY$ and define
	\begin{equation}\label{X-varphi}
	\begin{split}
	X_\varphi&:=\cos\varphi\cdot X + \sin\varphi\cdot Y\,\,\text{ and }\,\,
	Y_\varphi:=-\sin\varphi\cdot X + \cos\varphi\cdot Y.
	\end{split}
	\end{equation} 
	Then $L\in \mathcal O_\varphi$ if and only if $X_\varphi$ is invertible. 
	In this case, the graph matrix of $L$ with respect to $e^{i\varphi}L_B$ is $Y_\varphi X_\varphi^{-1}$.
	
\end{lemma}

\begin{proof}
	By the ${U}(n)$-invariance, $L$ is graphical over $e^{i\varphi}L_B$ if and only if $e^{-i\varphi}L$ is graphical over $L_B.$
	An oriented basis of $e^{-i\varphi}L$ is given by
	\begin{align*}
	e^{-i\varphi}(X+iY)
	& = \pr{\cos\varphi - i\sin\varphi}(X+iY) = X_\varphi
	+ iY_\varphi.
	\end{align*}
	Hence, the $n$-plane $e^{-i\varphi}L$ can be represented by 
	$X_\varphi+iY_\varphi.$
	Therefore, $e^{-i\varphi}L$ is graphical over $L_B$ if and only if $X_\varphi$ is invertible, and in such a case, the graph matrix of $L$ is $Y_\varphi X_\varphi^{-1}.$ 
\end{proof}

Next, we show that ${\rm Lag}^+(n)$ can be covered by finitely many such graph charts. 

\begin{lemma}
	\label{lem:graph-chart}
	Let $L\in {\rm Lag}^+(n).$
	Then there exists $\varphi\in\set{0,\frac{\pi}{n+1},\frac{2\pi}{n+1},\cdots,\frac{n\pi}{n+1}}$ such that $L\in\mathcal O_{\varphi}.$ Equivalently,
	\[
	\operatorname{Lag}^+(n)
	=
	\bigcup_{k=0}^{n}
	\mathcal O_{\frac{k\pi}{n+1}} .
	\]
\end{lemma}

\begin{proof}
	Suppose $L$ is represented by $U=X+iY$.
	Then for any $\varphi\in\bb R,$ we calculate
	\begin{align*}
	X_\varphi
	= {\rm Re }\, (e^{-i\varphi} U)
	= \frac 12 \pr{e^{-i\varphi} U + e^{i\varphi} \ovl U}
	= \frac 12 e^{-i\varphi} 
	\pr{U + z\ovl U}
	\end{align*}
	if we write $z=e^{2i\varphi}.$
	Therefore, there are at most $n$ values in $\bb R/\pi\bb Z$ of $\varphi$ such that $X_\varphi$ is singular. 
	Since the set $\set{0,\frac{\pi}{n+1},\frac{2\pi}{n+1},\cdots,\frac{n\pi}{n+1}}$ contains $n+1$ different values in $\bb R/\pi\bb Z,$ there exists at least one $\varphi$ in it such that $X_\varphi$ is invertible.
\end{proof}

\subsection{The function $\Lambda$}
\label{sec:Lambda}

In this section, a function $\Lambda$ will be defined on the oriented Lagrangian Grassmannian ${\rm Lag}^+(n).$ 
The definition is built from the natural action of a matrix on second exterior powers, which we recall first.

Let $S$ be an $n \times n$ matrix.
It induces a linear map
$\wedge^2 S : \wedge^2 \mathbb{R}^n
\to \wedge^2 \mathbb{R}^n$
defined on decomposable bivectors by
\[
\wedge^2 S(u \wedge v) := (Su) \wedge (Sv),
\]
and extended linearly to all of \(\wedge^2 \mathbb{R}^n\).
If \(S\) is diagonalizable with eigenvalues \(\lambda_1,\dots,\lambda_n\), then
\(\wedge^2 S\) is diagonalizable with eigenvalues
\begin{align}\label{evalue-transform}
\lambda_j\lambda_k
\,\text{ for }\,
1 \le j < k \le n .
\end{align}

\begin{definition}\label{Lambda}
	Let \(L \in \operatorname{Lag}^+(n)\), and choose a representative
	$X+iY\in U(n)$ for \(L\). 
	We define
	\begin{align}\label{Lambda-def}
	\Lambda(L)
	:= \det\pr{\wedge^2 X + \wedge^2 Y}.
	\end{align}
\end{definition}

This is a well-defined function because it is independent of the choice of representatives. 
More precisely, using the identification
\[
\operatorname{Lag}^+(n) \simeq {U}(n)/{SO}(n),
\]
a different representative of \(L\) is obtained from $X+iY$ by the right action of an element of ${SO}(n).$ 
This right ${SO}(n)$-action
does not change the determinant $\det\left(\wedge^2 X+\wedge^2 Y\right).$
Moreover, an explicit expression of $\Lambda$ in terms of its graph matrix is derived in the following lemma.

\begin{lemma}
	\label{lem:Lambda-in-terms-of-lambda}
	Suppose $n\ge 2$ and $L\in{\rm Lag}^+(n)$ is represented by $X+iY.$
	If $X$ is invertible and if $\lambda_k$'s are the eigenvalues of $YX^{-1},$ then
	\begin{align*}
	\Lambda(L)
	=\pr{\det X}^{n-1}
	\cdot \prod_{1\le j<k\le n} 
	\pr{1+\lambda_j\lambda_k}.
	\end{align*}
\end{lemma}

Thus, combined with Lemma~\ref{lem:detX>0}, $\Lambda$ can be completely expressed in terms of the eigenvalues~$\lambda_k$'s if the sign of $\det X$ is clear.

\begin{proof}
	By \eqref{Lambda-def} and the basic property of determinants, we have
	\begin{align*}
	\Lambda(L)
	= \det\pr{\pr{I + \wedge^2 YX^{-1}}\circ \wedge^2 X}
	& = \det\pr{\wedge^2X}
	\cdot \det\pr{I + \wedge^2 YX^{-1}}
	\end{align*}
	where $I$ is the identity operator on $\wedge^2\bb R^n.$
	The lemma then follows from the transformation law of eigenvalues mentioned in~\eqref{evalue-transform}.
\end{proof}

Besides the ${SO}(n)$-invariance, we will also use a second invariance of \(\Lambda\). 
Namely, \(\Lambda\) is invariant under the natural ${U}(1)$-subgroup of ${U}(n)$ generated by scalar unitary rotations. 
This is the circle direction associated with the Maslov class. 
The precise statement is recorded in the following lemma.

\begin{lemma}
	\label{lem:Lambda-rot-inv}
	For $n\ge 2,$ given $L\in{\rm Lag}^+(n)$ and $\varphi\in\bb R,$
	$\Lambda\pr{e^{-i\varphi}L}
	= \Lambda\pr{L}.$
\end{lemma}

\begin{proof} 
	Suppose $L$ is represented by $X+iY$, and hence $e^{-i\varphi}L$ is represented by $e^{-i\varphi}(X+iY)=X_\varphi+iY_\varphi$ where $X_\varphi$ and $Y_\varphi$ are defined in \eqref{X-varphi}.
	The conclusion of the lemma follows from the stronger relation that
	\begin{align}
	\label{wedge-rotation-invariance}
	\wedge^2 X_\varphi + \wedge^2 Y_\varphi=\wedge^2 X + \wedge^2 Y
	\end{align}
	as operatros on $\wedge^2\bb R^n.$
	To verify this, consider $u,v\in\bb R^n.$ 
	A direct calculation leads to
	\begin{align*}
	&\pr{\wedge^2 X_\varphi + \wedge^2 Y_\varphi}(u\wedge v)\\
	= & \pr{\cos\varphi\cdot X + \sin\varphi\cdot Y}u \wedge \pr{\cos\varphi\cdot X + \sin\varphi\cdot Y} v\\
	&+ \pr{-\sin\varphi\cdot X + \cos\varphi\cdot Y}u \wedge \pr{-\sin\varphi\cdot X + \cos\varphi\cdot Y}v\\
	= & \cos^2 \varphi \cdot Xu\wedge Xv
	+ \sin\varphi\cos\varphi\cdot Yu\wedge Xv
	+ \cos\varphi\sin\varphi\cdot Xu\wedge Yv
	+ \sin^2\varphi\cdot Yu\wedge Yv\\
	& + \sin^2 \varphi \cdot Xu\wedge Xv
	- \cos\varphi\sin\varphi\cdot Yu\wedge Xv
	- \sin\varphi\cos\varphi\cdot Xu\wedge Yv
	+ \cos^2\varphi\cdot Yu\wedge Yv\\
	= & Xu\wedge Xv + Yu\wedge Yv
	= \pr{\wedge^2 X + \wedge^2 Y}(u\wedge v).
	\end{align*}
	This proves \eqref{wedge-rotation-invariance} and implies the desired invariance.
\end{proof}

When $n=2,$ because the dimension of $\wedge^2\bb R^2$ is one, we can further simplify
\begin{align}
\label{Lambda-two-dim}
\Lambda\pr{L}= \det X + \det Y.
\end{align}
This is the same as the one defined in \eqref{Lambda_n=2}.

\subsection{Almost isoclinic Lagrangians} 

We now define the almost isoclinic region $\mathcal{AI}(n)\sbst{\rm Lag}^+(n).$ 
The definition of $\mathcal{AI}(n)$ depends on a fixed choice of the base $n$-plane $L_B\simeq\bb R^n.$ 
Such a choice determines a class in ${\rm Lag}^+(n)/{U}(1)\simeq {SU}(n)/{SO}(n).$

\begin{definition}\label{def:AI}
	Let $L\in{\rm Lag}^+(n)$ and fix \(\varphi \in \mathbb{R}\). 
	We choose a representative $X+iY\in U(n)$ of \(L\) and let \(X_{\varphi}\) and \(Y_{\varphi}\) denote the corresponding $\varphi$-rotated matrices defined in \eqref{X-varphi}. 
	Then $L\in\mathcal{AI}_{\varphi}(n)$ if the following conditions hold.
	\begin{enumerate}
		\item \label{AI-1}
		\(\det X_{\varphi}>0\);
		\item \label{AI-2}
		$1+\lambda_j\lambda_k>0$ for any $j$ and $k$ where $\lambda_j$'s are the eigenvalues of $Y_\varphi X_\varphi^{-1}.$
	\end{enumerate}
	We define
	\begin{align*}
	\mathcal{AI}(n)
	:=
	\bigcup_{\varphi \in [0,2\pi]}
	\mathcal{AI}_{\varphi}(n).
	\end{align*}
	Elements of \(\mathcal{AI}(n)\) are called {\bf almost isoclinic $n$-planes}.
\end{definition}

\begin{remark}
\label{rmk:general-matrices-for-AI}
    When verifying the conditions in Definition~\ref{def:AI}, one may work with $X+iY$ representing an arbitrary oriented basis, not necessarily an  orthonormal one.
    That is, it is enough to consider $X$ and $Y$ satisfying \eqref{Lag} alone. 
    A way of evaluating $\Lambda$ using an oriented basis is recorded in Lemma~\ref{Lambda_general}.
\end{remark}

If we let $\theta_k:=\arctan\lambda_k,$ then condition \eqref{AI-2} in the definition means
\begin{align*}
\cos(\theta_j-\theta_k)
= \frac{1+\lambda_j\lambda_k}{\sqrt{(1+\lambda_j^2)(1+\lambda_k^2)}}
>0,
\end{align*}
that is, $|\theta_j-\theta_k|<\pi/2,$ cf. \eqref{2d-theta}.
When $n=2,$ the region $\mathcal{AI}(2)$ is the same as the invariant regions defined in Section~\ref{sec:2D-AI}.

\begin{lemma}
	$\mathcal{AI}(2) 
	= \set{L\in{\rm Lag}^+(2): \Lambda(L)>0}.$
\end{lemma}

\begin{proof}
	We prove $\set{L\in{\rm Lag}^+(2): \Lambda(L)>0}\sbst \mathcal{AI}(2)$ as the other direction is clear.
	
	By $\Lambda>0$ and \eqref{Lambda-two-dim}, we know $\det X>0$ or $\det Y>0.$
	First, we assume $\det X>0.$
	In this case, we choose $\varphi=0.$ 
	Lemma~\ref{lem:Lambda-in-terms-of-lambda} then implies
	\begin{align*}
	0<\Lambda(L)
	= \det X \cdot \pr{1+\lambda_1\lambda_2}.
	\end{align*}
	Thus, it follows that  $1+\lambda_1\lambda_2>0,$
	so the conditions hold for $\varphi=0.$
	
	Next, we assume $\det Y>0.$
	In this case, we choose $\varphi=\pi/2,$ which implies
	$X_\varphi=Y$ and $Y_\varphi = -X$
	by \eqref{X-varphi}.
	Lemmas~\ref{lem:Lambda-in-terms-of-lambda} and~\ref{lem:Lambda-rot-inv} imply
	\begin{align*}
	0<\Lambda(L)
	= \Lambda\pr{e^{-i\varphi}L}
	& = \det X_\varphi \cdot \pr{1+\lambda_1\lambda_2}
	= \det Y \cdot \pr{1+\lambda_1\lambda_2}.
	\end{align*}
	This implies $1+\lambda_1\lambda_2>0,$ so the conditions hold for $\varphi=\pi/2.$
\end{proof}

The conditions in Definition~\ref{def:AI} are a natural generalization of those arising in the graphical setting. 
However, the key condition \eqref{AI-2} is not invariant under rotations. 
We point out that the region $\mathcal{AI}(n)$ can equivalently be characterized by a rotationally invariant condition, which appears more natural from a geometric perspective.

\begin{prop}
	\label{prop:AI-characterization}
	Let $n\ge 3.$
	Suppose $L\in{\rm Lag}^+(n)$ and  $X+iY$ is a representative of $L.$
	Then $L\in\mathcal{AI}(n)$ if and only if there exists $\psi\in\bb R$ such that the following conditions hold.
	\begin{enumerate}
		\item \label{AI-n-1}
		\(\Lambda(L)>0\);
		\item \label{AI-n-2}
		\(\det X_{\psi}>0\);
		\item \label{AI-n-3}
		$(1+\lambda_j\lambda_k)(1+\lambda_k\lambda_\ell)
		(1+\lambda_\ell\lambda_j)>0$
		for any $j$, $k,$ and $\ell$ where $\lambda_j$'s are the eigenvalues of $Y_\psi X_\psi^{-1}.$
	\end{enumerate}
	In particular, 
	$\mathcal{AI}(3) = \set{L\in{\rm Lag}^+(3): \Lambda(L)>0}.$
\end{prop}

We emphasize that the angle $\psi$ in the proposition may differ from the angle $\varphi$ in Definition~\ref{def:AI}.
By Lemma~\ref{lem:lambda-rotate}, condition~\eqref{AI-n-3} is invariant under rotations that preserves condition~\eqref{AI-n-2}.
That is, if $L\in\mathcal O_{\varphi}\cap \mathcal O_{\td\varphi}$ with $\det X_{\varphi}$ and $\det X_{\td\varphi}$ positive, then condition \eqref{AI-n-3} is true for $\varphi$ if and only if it is true for $\td\varphi.$
Moreover, in the proof our main results (e.g., Theorem~\ref{thm:main-thm-log-concavity} and Proposition~\ref{prop:Hess-n-dim}), the seemingly weaker condition~\eqref{AI-n-3} in Proposition~\ref{prop:AI-characterization} appears naturally in the calculations.

\begin{proof}
	The ``only if'' direction is clear based on Lemmas~\ref{lem:detX>0}, \ref{lem:Lambda-in-terms-of-lambda}, and~\ref{lem:Lambda-rot-inv}.
	Thus, it remains to prove the ``if'' direction.
	Assume $\lambda_j$'s are the eigenvalues of $Y_\psi X_\psi^{-1}.$
	
	Let $q_{jk}:=1+\lambda_j\lambda_k$ and let
	$s_j:={\rm sign}(q_{1j})
	= {\rm sign} (1+\lambda_1\lambda_j)
	\in\set{\pm 1}.$
	Since $q_{jk}q_{1j}q_{1k}>0$ for any $j<k<\ell,$ it follows that
	$1 = {\rm sign}(q_{jk})\cdot s_j\cdot s_k.$
	This means
	${\rm sign}(q_{jk})=s_js_k$
	for any $j$ and $k.$

	Consider $\mathcal P
	:= \set{j: s_j=1}$
    and $\mathcal N
	:= \set{k: s_k=-1}.$
	If $\mathcal N=\emptyset,$ then we can simply take $\psi=\varphi.$
	Thus, in the rest of the proof, we assume $\mathcal N\neq\emptyset.$ 
	This means that $\lambda_j\lambda_k<-1<0$ for any $j\in\mathcal P$ and $k\in\mathcal N.$
	In particular, since $\mathcal N$ is non-empty, this implies that all $\lambda_j$'s in $\mathcal P$ have the same sign, and so do all $\lambda_k$'s in $\mathcal N.$
	We will show that in this case, onw of the angles $\psi \pm \pi/2$ will satisfy all the conditions in the proposition.
	
	First, we observe that $\lambda_\ell\neq 0$ for all $\ell.$
	In fact, if there existed $\ell$ such that $\lambda_\ell=0,$ then we could take $k\in\mathcal N$ and obtain
	${\rm sign}\pr{q_{1k}q_{k\ell}q_{1\ell}}
	= s_k\cdot 1\cdot 1<0,$
	contradicting condition~\eqref{AI-n-3} in the proposition.
	This proves the claim that $\lambda_\ell\neq 0$ for all $\ell.$
	In particular, this implies $Y_\psi X_\psi^{-1}$ is invertible, and hence $Y_\psi$ is invertible.
	
	We look at condition~\eqref{AI-2} for the angle $\psi+\pi/2.$
	Note that by~\eqref{X-varphi}, 
	$X_{\psi+\pi/2} 
	= Y_\psi,$
	which is invertible.
	Thus, let $\td\lambda_j$'s be the eigenvalues of $Y_{\psi+\pi/2} X_{\psi+\pi/2}^{-1}.$
	Lemma~\ref{lem:lambda-rotate} implies	
	\begin{align*} {\rm sign} \pr{1+\td\lambda_j\td\lambda_k}
	= {\rm sign} \pr{\frac{q_{jk}}{\lambda_j\lambda_k}}
	= {\rm sign} \pr{\frac{s_js_k}{\lambda_j\lambda_k}}.
	\end{align*}
	Fix two indices $j$ and $k.$
	If they both lie in either $\mathcal P$ or $\mathcal N,$ then $\lambda_j$ and $\lambda_k$ have the same sign.
	If $j\in\mathcal P$ and $k\in\mathcal N,$ then $\lambda_j$ and $\lambda_k$ have different signs.
    Thus, in both cases, $(s_js_k)/(\lambda_j\lambda_k)>0,$ so $1+\td\lambda_j\td\lambda_k
	>0.$

	It remains to check condition~\eqref{AI-1}.
	If $n$ is odd, then we choose $\varphi=\psi+\pi/2$ when $\det Y_\psi>0$ and we choose $\varphi=\psi-\pi/2$ when $\det Y_\psi>0.$
	This does not change condition~\eqref{AI-n-3} by Lemma~\ref{lem:lambda-rotate} again.
	If $n$ is even, Lemma~\ref{lem:Lambda-in-terms-of-lambda} implies
	\begin{align*}
	0<\Lambda(L)
	= \pr{\det X_{\psi+\pi/2}}^{n-1} \cdot \prod_{1\le j<k\le n} 
	\pr{1+\td\lambda_j \td\lambda_k}.
	\end{align*}
	Since $1+\td\lambda_j\td\lambda_k>0$ for all $j$ and $k$ and $n-1$ is odd, it follows that $\det X_{\psi+\pi/2}>0,$ so $\varphi=\psi+\pi/2$ satisfies all the conditions.
	This finishes the proof.
\end{proof}

We remark that conditions \eqref{AI-n-2} and \eqref{AI-n-3} do not imply \eqref{AI-n-1} in Proposition~\ref{prop:AI-characterization}.
We finish the section by defining almost isoclinic Lagrangian submanifolds.

\begin{definition}
	Let $M$ be an $n$-dimensional oriented Lagrangian submanifold in $\bb R^{2n}.$
	We say that $M$ is an {\bf almost isoclinic Lagrangian} if for any $p\in M,$ the oriented tangent space $T_pM$ is almost isoclinic, i.e. $T_pM\in \mathcal{AI}(n)$.
\end{definition}

\subsection{Another expression for  $\Lambda$}

We provide an alternative expression for the quantity $\Lambda$ in terms of the complex coordinates on $\mathbb{C}^n$.

Let $z_k$, $k=1,\dots,n$, denote the standard complex coordinates on $\mathbb{C}^n$.
Suppose $L \subset \mathbb{C}^n$ is an oriented Lagrangian subspace represented by an oriented orthonormal basis $\set{e_k}$. Such a basis is determined up to the action of ${SO}(n)$.
The $n\times n$ matrix
$\big[\,dz_k(e_\ell)\,\big]_{1\le k,\ell\le n}$
is an element of ${U}(n),$ and the $\binom{n}{2}\times \binom{n}{2}$ real matrix
$$\left[\, \operatorname{Re}\!\left(dz_k \wedge d\bar{z}_\ell\right)(e_p, e_q) \,\right]_{\substack{1\le k<\ell\le n \\ 1\le p<q\le n}}$$
provides a way to generalize the formula for $\Lambda$ in \eqref{Lambda_n=2}.

\begin{lemma}
   Let $L$ be an oriented Lagrangian subspace of $\mathbb{C}^n$. Then
    \begin{equation*}
    \Lambda(L)
    =
    \det
    \left[\, \operatorname{Re}\!\left(dz_k \wedge d\bar{z}_\ell\right)(e_p, e_q) \,\right]_{\substack{1\le k<\ell\le n \\ 1\le p<q\le n}}
    \end{equation*}
    is well-defined. 
    That is, it is independent of the choice of an oriented orthonormal basis $\set{e_k}$ of $L$.
\end{lemma}

The lemma follows from the definition of $\Lambda.$
As a comparison, the Lagrangian angle of $L$ is defined to be the argument of  
$\pr{dz_1\wedge \cdots \wedge dz_n}(e_1, \cdots, e_n),$
which is also independent of the choice of an oriented orthonormal basis.

\section{\bf Log-concavity of $\Lambda$ and its consequence}
\label{sec:log-concavity}

In this section, we introduce the Grassmannian metric and explain the log-concavity property of the function $\Lambda$ on $\mathcal{AI}(n).$
This result may be of independent interest. 
For our purposes, we will use it to estimate the Laplacian $\D$ of $\log\Lambda$ on a minimal Lagrangian; see Theorem~\ref{thm:Lambda-elliptic-equation}.
This estimate will be the main ingredient for the proof of the Bernstein-type theorems in Section~\ref{sec:Bernstein}.

\subsection{Hessians of $\log\Lambda$ on ${\rm Lag}^+(n)$}

We first introduce the Grassmannian metric, with respect to which we will evaluate the Hessian of $\log \Lambda$ on the almost isoclinic region.
In this section, we will use $L_0$ to denote a rotation of a fixed base Lagrangian $n$-plane $L_B\simeq\bb R^n.$
That is, each $L_0$ will represent $e^{i\varphi}\bb R^n$ and we will let $\mathcal O_{L_0}:=\mathcal O_\varphi$ with this assumption understood.

Let $G(n,2n)$ be the Grassmannian of $n$-dimensional linear subspaces of $\bb R^{2n}.$
We endow $G(n,2n)$ with the canonical Grassmannian metric.
That is, given $L\in G(n,2n),$ we identify 
$T_LG(n,2n)\simeq {\rm Hom}\pr{L,L^\perp}$
on which the metric is defined by
\begin{align}\label{Gr-metric}
\pair{S_1,S_2}_{\rm Gr}
:=\tr\, S_1^TS_2
\end{align}
for any two tangent vectors $S_1$ and $S_2.$
The Grassmannian metric induces a Riemannian metric on~${\rm Lag}^+(n).$

When $L\in{\rm Lag}^+(n),$ we identify $L^\perp$ with $JL,$ where $J$ is the standard complex structure on $\bb C^n\simeq\bb R^{2n}.$
The tangent space of ${\rm Lag}^+(n)$ is the space of self-adjoint elements in ${\rm Hom}(L,JL).$
We fix $L_0\in {\rm Lag}^+(n)$ and an oriented orthonormal basis of it. 
If we consider the graph chart $\mathcal O_{L_0}$ with respect to $L_0,$ any $L\in\mathcal O_{L_0}$ can be represented by a symmetric matrix $S$ if one views $L=L_S=\set{x+JSx: x\in L_0}.$
With respect to the orthonormal basis of $L_0$, a tangent vector $\alpha\in T_{L_S}G(n,2n)$ is given by a symmetric matrix $A$ via a variation $S(t)=S+tA.$
Hence, at a point $L_S\in{\rm Lag}^+(n),$ one can locally evaluate the differential and the Hessian of a function on ${\rm Lag}^+(n)$ by symmetric matrices $S$ and $A.$

Theorem~\ref{thm:main-thm-log-concavity} is stated with these identifications understood.
Indeed, we derive an explicit formula of $\operatorname{Hess}\big(\log \Lambda(S)\big)(A, A).$  
In the following proposition, we work on a graph chart with respect to which the conditions in Definition~\ref{def:AI} hold.

\begin{prop}
	\label{prop:Hess-n-dim}
	Suppose $L\in \mathcal{AI}(n),$ and suppose in a graph chart based on $L_0,$ $L$ is given by the graph of a symmetric matrix $S.$
	If we choose an orthonormal basis in which $S={\rm diag}(\lambda_1,\cdots,\lambda_n)$, then given any symmetric $A=\pr{a_{jk}}_{j,k}$ in the same basis, we have
	\begin{align*}
	&\operatorname{Hess}\big(\log \Lambda(S)\big)(A, A)\\
	= & 
	-\sum_{1\le j<k\le n} \pr{\frac 1{(1+\lambda_j\lambda_k)^2} 
		\pr{\sqrt{\frac{1+\lambda_k^2}{1+\lambda_j^2}}a_{jj} - \sqrt{\frac{1+\lambda_j^2}{1+\lambda_k^2}}a_{kk}}^2
		+ \frac{4}{\pr{1+\lambda_j^2}\pr{1+\lambda_k^2}} a_{jk}^2}\\
	& - \sum_{1\le j<k\le n} 
	\sum_{\substack{1\le m\le n\\ m\neq j,k}} 
	\frac{2(1+\lambda_j\lambda_k)(1+\lambda_m^2)}
	{(1+\lambda_j^2)(1+\lambda_k^2)(1+\lambda_j\lambda_m)(1+\lambda_k\lambda_m)}
	a_{jk}^2.
	\end{align*}
\end{prop}
The right hand side is a quadratic form in $a_{jk}$, with coefficients depending on $S$.
In a graph chart in which the conditions in Definition~\ref{def:AI} hold for $S$, this quadratic form is manifestly non-positive.
The proof of Proposition~\ref{prop:Hess-n-dim} will be given in Section~\ref{sec:proof-Lambda}.

\begin{remark}
	\label{rmk:Hess-different-basis}
	The formula can be expressed in a simpler form.
	Following the notations in Proposition~\ref{prop:Hess-n-dim}, if another matrix $\td A=\pr{\td a_{jk}}_{j,k}$ is defined such that
	\begin{align*}
	a_{jk}
	= \td a_{jk} \sqrt{1+\lambda_j^2} \sqrt{1+\lambda_k^2},
	\end{align*}
	then the Hessian of $\log\Lambda$ at $S$ in the direction $(A,A)$ can be expressed as
	\begin{align*}
	-\sum_{1\le j<k\le n} 
	\pr{
	\frac{(1+\lambda_j^2)(1+\lambda_k^2)}{(1+\lambda_j\lambda_k)^2}\pr{\td a_{jj}-\td a_{kk}}^2
	+ 4\,\td a_{jk}^2
	}
	- \sum_{1\le j<k\le n} 
	\sum_{\substack{1\le m\le n\\ m\neq j,k}} 
	\frac{2(1+\lambda_j\lambda_k)(1+\lambda_m^2)}
	{(1+\lambda_j\lambda_m)(1+\lambda_k\lambda_m)}
	\td a_{jk}^2.
	\end{align*}
	In the next section, we will see that this matrix representation $\td A$ is another natural way to express an element in $T_{L_S}{\rm Lag}^+(n)$ with respect to another basis.
\end{remark}

\subsection{Consequences on almost isoclinic minimal Lagrangians}
\label{sec:Gauss-map}

The main application of the log-concavity property of the function $\Lambda$ is an estimate for $\D\log\Lambda$ on an almost isoclinic minimal Lagrangian.
In the following estimate, we let $A$ be the second fundamental form of a Lagrangian submanifold $\Sigma$ in $\bb R^{2n}$ and we let $h_{jk\ell}:=\pair{A(e_j,e_k),Je_{\ell}}$ be the component of $A.$
When the Lagrangian $\Sigma$ is oriented, we use the same notation $\Lambda$ to denote the function $\Lambda(p):=\Lambda\pr{T_p\Sigma}$ for $p\in\Sigma.$
As a result, if $\Sigma$ is almost isoclinic, then $\Lambda$ is a positive function on $\Sigma.$

\begin{thm}
	\label{thm:Lambda-elliptic-equation}
	Let $\Sigma$ be an almost isoclinic Lagrangian minimal submanifold in $\bb R^{2n}.$
	Then
	\begin{align*}
	\D\log\Lambda
	\le - 4
	\sum_{\substack{1\le \ell\le n\\ 1\le j<k\le n}} h_{jk\ell}^2.
	\end{align*}
	In particular, for any $n\ge 2,$ there exists $c_n>0$ such that
	\begin{align*}
	\D\log\Lambda \le -c_n|A|^2,
	\end{align*}
	and when $n=2,$ $c_2$ can be chosen to be $1.$
\end{thm}

\begin{remark}
    The log-concavity property of $\Lambda$ also has an immediate consequence for a mean curvature flow of almost isoclinic Lagrangians.
    In fact, combining Proposition~\ref{prop:Hess-n-dim} and the main result in \cite{W03-MCF} leads to the estimate $(\bd_t-\D)\Lambda\ge c_n|A|^2$ along an almost isoclinic Lagrangian mean curvature flow.
    This estimate will be used in future work about regularity properties of such a flow.
\end{remark}

We will prove the theorem at the end of this section.
The main ingredients of the proof include the log-concavity of $\Lambda$ (Proposition~\ref{prop:Hess-n-dim}), the harmonicity of the Gauss map of a minimal immersion, and a correct identification of different bases of the tangent space of ${\rm Lag}^+(n).$

We first explain how the analytic property of the Gauss map helps us estimate the Laplacian of a function on a minimal submanifold.
Given an immersion $\Sigma^n\to\bb R^{2n},$ we let $g$ be the metric on $\Sigma$ induced by the Euclidean metric. 
Ruh--Vilms~\cite{RV} showed that the Gauss map $\gamma\colon \pr{\Sigma,g}\to G(n,2n)$ of the immersion is harmonic whenever the immersion is minimal.
That is,
\begin{align}\label{Minimal-Gauss}
\tr\, \n d\gamma = 0,
\end{align}
where $d\gamma$ is considered as a section of $T^*\Sigma\otimes \gamma^{-1} TG(n,2n)$ and the trace is with respect to $g.$

From now on, we will specialize to the case when $\Sigma$ is an oriented minimal Lagrangian and hence its Gauss map is valued in ${\rm Lag}^+(n)\sbst G(n,2n).$
Given a smooth function $F\colon U\to \bb R$ on an open set $U\sbst {\rm Lag}^+ (n),$ if $\gamma(\Sigma)\sbst U,$ one can consider the function $f:=F\circ \gamma$ on $\Sigma.$ 
The chain rule implies
\begin{align*}
\Hess_f(\cdot,\cdot)
& = \Hess_F \pr{d\gamma(\cdot), d\gamma (\cdot)}
+ dF\pr{\pr{\n d\gamma}(\cdot,\cdot)}
\end{align*}
where $\Hess_f$ is the Hessian of $f$ with respect to $g$ and $\Hess_F$ is the Hessian of $F$ with respect to the Grassmannian metric.
We then derive from \eqref{Minimal-Gauss} that 
\begin{equation}\label{evolution-equation}
\begin{split}    
\D f
= \tr\pr{\Hess_f}
= \sum_{j=1}^n \Hess_F \pr{d\gamma(e_j), d\gamma (e_j)} + dF\pr{\tr\, \n d\gamma}
& = \sum_{j=1}^n \Hess_F \pr{d\gamma(e_j), d\gamma (e_j)}
\end{split}
\end{equation}
if $e_j$'s form a local orthonormal tangent frame on $\Sigma.$
Therefore, it suffices to evaluate the Hessian of the function $F$ on $U$ in the Grassmannian metric.

To evaluate the Hessian in \eqref{evolution-equation}, it is essential to correctly interpret $d\gamma(e_j),$ as there are two natural matrix representations on $T_{L_S}{\rm Lag}^+(n).$
Since the transformation formula will also be used later, we briefly explain their relation.

Recall that when we fix $L_0\in {\rm Lag}^+(n)$ and an oriented orthonormal basis of it, any $L\in\mathcal O_{L_0}$ can be represented by a symmetric matrix $S$ if one views $L=L_S=\set{x+JSx: x\in L_0}.$
With respect to the basis of $L_0$, a tangent vector $\alpha\in T_{L_S}{\rm Lag}^+(n)$ is given by a symmetric matrix $A$ via a variation $S(t)=S+tA.$
We will denote this matrix by $A=[\alpha]_{L_0}.$

Alternatively, given such an element $\alpha\in T_{L_S}{\rm Lag}^+(n),$ via 
$$T_{L_S}{\rm Lag}^+(n)\sbst T_{L_S}G(n,2n)\simeq {\rm Hom}\pr{L_S,L_S^\perp},$$ 
$\alpha$ can be represented by another matrix when we use the corresponding bases of $L_S$ and $JL_S$ induced from those of $L_0$ and $JL_0.$ 
We will denote this matrix by $[\alpha]_{L_S}.$ 
Based on the formula for projections onto linear subspaces, these two matrices are related by
\begin{align}\label{matrix-in-graph-chart}
[\alpha]_{L_0}
= \pr{1+S^2}^{1/2} [\alpha]_{L_S} \pr{1+S^2}^{1/2}.
\end{align}
The newly expressed formula in Remark~\ref{rmk:Hess-different-basis} is indeed written using the matrix representation $[\alpha]_{L_S}.$

We apply \eqref{matrix-in-graph-chart} to our setting.
Fix an orthonormal frame $\set{e_j}$ at a point $p\in \Sigma$ where $(\Sigma,g)\to\bb R^{2n}$ is a minimal Lagrangian immersion and $\gamma\colon \Sigma\to {\rm Lag}^+(n)$ is its Gauss map.
Suppose $T_p\Sigma = L_S$ is in a graph chart $\mathcal O_{L_0}.$
For any $\ell=1,\cdots, n$ and $v\in L_S,$
\begin{align*}
d\gamma\pr{e_\ell}(v)
= A(e_\ell,v)\in N_pM = L_S^\perp
\end{align*}
where $A$ is the second fundamental form of $\Sigma.$
If we let $h_{jk\ell}:=\pair{A(e_j,e_k),Je_{\ell}}_g,$ then the matrix representation $[d\gamma(e_\ell)]_{L_S}$ is given by $\pr{h_{jk\ell}}_{j,k}.$
Thus, by \eqref{matrix-in-graph-chart}, the representing matrix in the graph chart $\mathcal O_{L_0}$ is
\begin{align}\label{d-gamma-ej}
[d\gamma\pr{e_\ell}]_{L_0}
= \pr{1+S^2}^{1/2} \pr{h_{jk\ell}}_{j,k} \pr{1+S^2}^{1/2}.
\end{align} 
We combine this identification with the log-concavity of $\Lambda$ and the harmonicity of $\gamma$ to prove Theorem~\ref{thm:Lambda-elliptic-equation}.

\begin{proof}
	[Proof of Theorem~\ref{thm:Lambda-elliptic-equation}]
	Fix a point $p\in\Sigma$ and find a base $n$-plane $L_0$ such that the tangent plane $T_p\Sigma$ can be written as the graph $L_S$ of a symmetric matrix $S$ over $L_0$ and it satisfies the conditions in Definition~\ref{def:AI}.
	Fix a basis of $L_0$ with respect to which $S={\rm diag}\pr{\lambda_1,\cdots,\lambda_n}.$ 
    
	Let $e_\ell$'s denote the corresponding orthonormal basis of $T_p\Sigma$ and  $h_{jk\ell}:=\pair{A(e_j,e_k),Je_{\ell}}_g.$
	The relation~\eqref{d-gamma-ej} implies that the $(j,k)$-entry of $[d\gamma\pr{e_j}]_{L_0}$ is
    $h_{jk\ell}
	\cdot \sqrt{1+\lambda_j^2}
	\cdot \sqrt{1+\lambda_k^2}.$
	Combining this with Proposition~\ref{prop:Hess-n-dim} (cf. Remark~\ref{rmk:Hess-different-basis}) and~\eqref{evolution-equation} leads to
	\begin{align*}
	\D\log\Lambda
	= & 
	-\sum_\ell \sum_{j<k} \pr{ \frac{\pr{1+\lambda_j^2}\pr{1+\lambda_k^2}}
		{(1+\lambda_j\lambda_k)^2} 
		\pr{h_{jj\ell}-h_{kk\ell}}^2
		+ 4 h_{jk\ell}^2}\\
	& - \sum_\ell\sum_{j<k} 
	\sum_{m\neq j,k} 
	\frac{2(1+\lambda_j\lambda_k)(1+\lambda_m^2)}
	{(1+\lambda_j\lambda_m)(1+\lambda_k\lambda_m)}
	h_{jk\ell}^2\\
	\le & -4\sum_\ell\sum_{j<k}h_{jk\ell}^2.
	\end{align*}
	
	Next, we prove the existence of a dimensional constant $c_n.$
	We note that all the terms $h_{jk\ell}$'s appear in the sum above except those of the form $h_{jjj}.$
	However, since $\Sigma$ is minimal, we know
	\begin{align*}
	h_{jjj}^2
	= \pr{-\sum_{k=1}^n h_{kkj}}^2,
	\end{align*}
	and hence there exists a positive number $c_n>0$ such that
	\begin{align*}
	\sum_\ell\sum_{j<k}h_{jk\ell}^2
	\ge c_n \sum_{j,k,\ell} h_{jk\ell}^2
	= c_n|A|^2.
	\end{align*}
	When $n=2,$ the sum is simply
	\begin{align*}
	4\sum_\ell\sum_{j<k}h_{jk\ell}^2
	= 4h_{112}^2 + 4h_{212}^2
	= \pr{-h_{222}}^2 + 3 h_{112}^2
	+ 3h_{221}^2 + \pr{-h_{111}}^2
	= |A|^2,
	\end{align*}
    based on the symmetries of $h_{jk\ell}$ in all three indices $j,k,$ and $\ell$ and the minimality condition $h_{111}+h_{221}=h_{112}+h_{222}=0.$
	This finishes the proof of the theorem.
\end{proof}

\section{\bf Examples of almost isoclinic Lagrangian submanifolds}
\label{sec:ex}

In this section, we will provide various examples of almost isoclinic Lagrangian submanifolds.

\subsection{Graphical examples}

We begin with graphical examples.
From Definition~\ref{def:AI}, there is a simple criterion for a gradient graph to be almost isoclinic.
We record it in the following lemma.

\begin{lemma}
\label{lem:AI-graph}
	Let $U$ be an open set in $\bb R^n$ and $u\colon U\to\bb R$ be a $C^2$ function.
	Suppose $\lambda_j$'s are the eigenvalues of $\n^2 u.$
	If $1+\lambda_j\lambda_k>0$ for any $j$ and $k,$ then the graph $\Sigma$ of $\n u$ is an almost isoclinic Lagrangian.
\end{lemma}

\begin{proof}
	We take $L_B=\bb R^n$ to be the fixed base $n$-plane.
	Since $\n^2 u$ is a symmetric matrix, we could find $Q\in{SO}(n)$ such that $\Hess_u = QDQ^T,$ where $D={\rm diag}(\lambda_1,\cdots,\lambda_n).$
	Then the tangent space at a point can be represented by 
	\begin{align*}
	X &= Q\cdot {\rm diag}\pr{\frac 1{\sqrt{1+\lambda_1^2}}, \cdots, \frac 1{\sqrt{1+\lambda_n^2}}}\,
	\text{ and }\,
	Y =  Q\cdot {\rm diag}\pr{\frac{\lambda_1}{\sqrt{1+\lambda_1^2}}, \cdots, \frac{\lambda_n}{\sqrt{1+\lambda_n^2}}}.
	\end{align*}
	Thus, $\det X>0$ and the eigenvalues of $YX^{-1}$ are
	$\lambda_1,\cdots,\lambda_n.$
	Since $1+\lambda_j\lambda_k>0$ for any $j$ and $k,$ by Definition~\ref{def:AI}, $\Sigma$ is almost isoclinic.	
\end{proof}

As a result, we obtain examples of almost isoclinic Lagrangian submanifolds by choosing $C^2$ functions with conditions on the eigenvalues of their Hessians.
A particular consequence is a criterion for almost isoclinicity of a Lagrangian $n$-plane given by the gradient graph of a quadratic polynomial.

\begin{cor}
	\label{cor:AI-graph}
	Let $u\colon \bb R^n\to\bb R$ be a quadratic polynomial of the form
	\begin{align*}
	u(x_1,\cdots,x_n)
	= \frac 12 \sum_{j=1}^n \lambda_j x_j^2
	\end{align*} 
	for some real numbers $\lambda_j$'s.
	If $1+\lambda_j\lambda_k>0$ for any $j$ and $k,$ then the graph of $\n u$ is an almost isoclinic Lagrangian.
\end{cor}

Given a $C^2$ function $u,$ when all eigenvalues of $\n^2 u$ are positive, the graph of $\n u$ is called a {\bf convex} Lagrangian \cite{Y02,SW02,CCH}.
When any two eigenvalues $\lambda_j$ and $\lambda_k$ of $\n^2 u$ satisfy $1+\lambda_j\lambda_k>0$ and $\lambda_j+\lambda_k>0,$ the graph of $\n u$ is called a {\bf two-convex} Lagrangian \cite{TTW25}.
Based on Lemma~\ref{lem:AI-graph} and Corollary~\ref{cor:AI-graph}, the notion of almost isoclinicity thus naturally generalizes the previous positivity conditions in the graphical setting.

\begin{figure}[h]
	\centering
	\includegraphics[width=5cm]{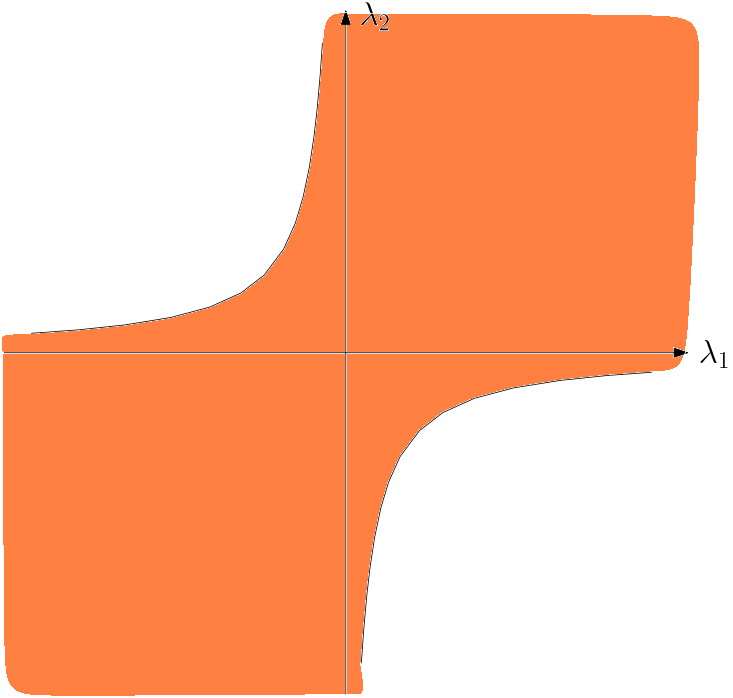}
	\includegraphics[width=5cm]{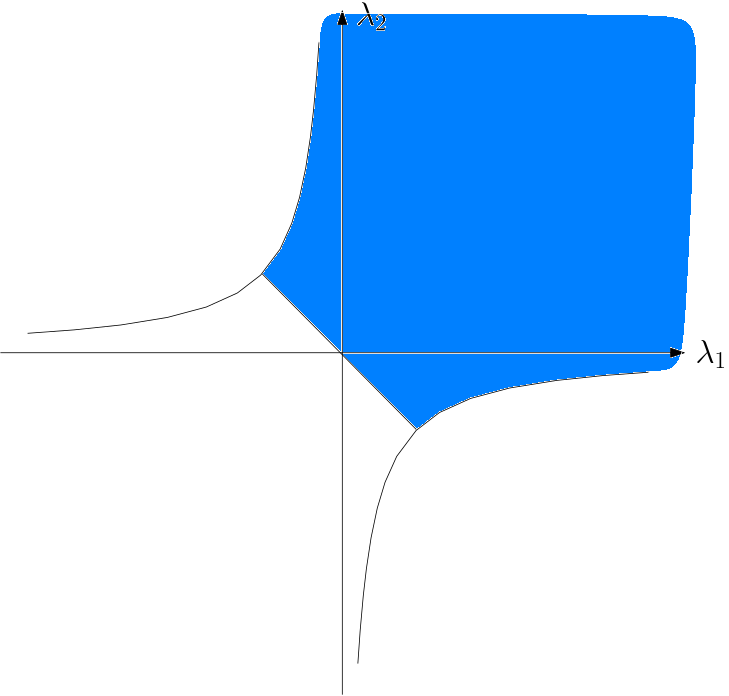}
	\caption{The orange region shows the range of $\lambda_1$ and $\lambda_2$ for graphical almost isoclinic Lagrangian submanifolds (i.e., $1+\lambda_1\lambda_2>0$).
    The blue region shows the range of $\lambda_1$ and $\lambda_2$ for graphical two-convex Lagrangian submanifolds (i.e., both $1+\lambda_1\lambda_2>0$ and $\lambda_1+\lambda_2>0$).
    }
	\label{pic:regions}
\end{figure}

\subsection{Symmetric examples}\label{sym-exam}
The next class of examples has an ${SO}(n)$-symmetry.
With the help of the symmetry, the analysis of the submanifolds could be reduced to that of one-dimensional curves.
As a consequence, we construct examples that are not graphical over any rotated $n$-planes.

We explain the construction as follows.
Suppose $Z=\set{z(r)\in\bb C:r\in I}$ is an immersed curve in $\bb C\simeq \bb R^2$ parametrized by $r\in I\sbst\bb R.$
We write
\begin{align*}
z(r) = \rho(r) e^{i\phi(r)}
= x(r) + iy(r)
\end{align*}
with $x=\rho\cos\phi$ and $y=\rho\sin\phi.$
We then look at the immersion $F\colon I\times \mathbb S^{n-1}\to\bb R^{2n}$ given by
\begin{align}\label{F-sym-def}
F(r,u)
= F(r,u_1,\cdots,u_n)
= \pr{x(r)\,u_1,
	\cdots ,
	x(r)\,u_n,
	y(r)\,u_1,
	\cdots,
	y(r)\,u_n}
\end{align}
where $u=(u_1,\cdots,u_n)\in \mathbb S^{n-1}$ is a parametrization of the standard $(n-1)$-dimensional sphere.
Then $F$ defines a Lagrangian submanifold;
see \cite{N07, S20}.
In the rest of this section, we will call such a Lagrangian submanifold an {\bf ${SO}(n)$-symmetric Lagrangian}.

The main criterion that will be applied later is summarized in the following proposition.
The proposition shows that almost isoclinicity is equivalent to monotonicity of the radial function $\rho.$

\begin{prop}
	\label{prop:AI-sym}
	Let $\rho$ and $\phi$ be $C^2$ functions on an interval $I,$ and let $F\colon I\times \mathbb S^{n-1}\to\bb R^{2n}$ be the immersed Lagrangian submanifold defined by $\rho$ and $\phi$ as in \eqref{F-sym-def}.
	Suppose $\rho>0.$
	\begin{enumerate}
		\item \label{AI-sym-1}
		If $\rho'>0,$ then $F$ defines an almost isoclinic Lagrangian.
		\item \label{AI-sym-2}
		If $F$ is an almost isoclinic Lagrangian, then $\rho'\neq 0.$
	\end{enumerate}
    Moreover, the function $\Lambda$ on $F$ is given by
    \begin{equation}\label{Lambda-formula-sym}
    \Lambda
    = \pr{\frac{\rho'}{\sqrt{\pr{\rho'}^2 + \rho^2\pr{\phi'}^2}}}^{n-1}.
    \end{equation}
\end{prop}

We present two lemmas that will be used in the proof of the proposition.
We first need a formula for $\Lambda$ in terms of bases that are not necessarily orthonormal. The following lemma generalizes the previous formula for $\Lambda$, which applies in the special case when the basis matrix is in ${U}(n)$.

\begin{lemma}\label{Lambda_general}
Suppose an oriented Lagrangian subspace $L$ is represented by a $2n\times n$ matrix
$
\begin{pmatrix}
A\\B
\end{pmatrix}.
$
Then
\[
\Lambda(L)
=
\frac{
\det\!\left(\wedge^2 A+\wedge^2 B\right)
}{
\det\!\left(A^TA + B^TB\right)^{(n-1)/2}
}.
\]
\end{lemma}

We note that the denominator is always positive as long as the columns of $
\begin{pmatrix}
A\\B
\end{pmatrix}
$ are linearly independent. 
Therefore, $\Lambda(L)$ and $\det\!\left(\wedge^2 A+\wedge^2 B\right)$ always have the same sign.

\begin{proof}
[Proof of Lemma~\ref{Lambda_general}]
	Let $X+iY\in U(n)$ be a representative of $L$ by an oriented orthonormal basis of it.
	Then there exists an element~$Q\in {GL}(n)$ with $\det Q>0$ such that
	\begin{align*}
	\begin{pmatrix}
	A\\B
	\end{pmatrix}Q
	= \begin{pmatrix}
	X\\Y
	\end{pmatrix};
	\end{align*}
	that is, $X=AQ$ and $Y=BQ.$
	Then the definition of $\Lambda$ implies
	\begin{equation}\label{L-AB}
	\begin{split}
	\Lambda(L)
	= \det\pr{\wedge^2 X + \wedge^2 Y}
	= \det\pr{\wedge^2 (AQ) + \wedge^2 (BQ)}
	& = \det\pr{\pr{\wedge^2 A + \wedge^2 B}\circ \wedge^2Q}\\
	& = \det\pr{\wedge^2 A + \wedge^2 B}
	\cdot \det \pr{\wedge^2 Q}\\
	& = \det\pr{\wedge^2 A + \wedge^2 B}
	\cdot \pr{\det Q}^{n-1}.
	\end{split}
	\end{equation}
	On the other hand, the orthogonality condition~\eqref{ortho} implies
	\begin{align*}
	I = X^TX + Y^TY
	= (AQ)^T(AQ) + (BQ)^T(BQ)
	= Q^T\pr{A^TA+B^TB}Q.
	\end{align*}
	Since $\det Q>0,$ this implies
	\begin{align*}
	\det Q = \pr{\det\pr{A^TA+B^TB}}^{-1/2}.
	\end{align*}
	Combining this with \eqref{L-AB} finishes the proof.
\end{proof}

The following linear algebra lemma will be used repeatedly. 

\begin{lemma}\label{aI+buu^T}
Let \(u\in \mathbb R^n\) be a unit vector and $a$ and $b$ be real numbers.
\begin{enumerate}
\item[(i)] For $A=aI+buu^T,$ 
$\det A=(a+b)a^{n-1}$ and
$\det(\wedge^2 A) =(a+b)^{n-1}\cdot a^{(n-1)^2}.$

\item[(ii)] The matrix \(A=aI + buu^T\) is invertible if and only if $a\neq 0$ and $a+b\neq 0.$
In this case,
\[
A^{-1}
=
\frac{1}{a}I-\frac{b}{a(a+b)}uu^T.
\]

\item[(iii)] If $A_1=a_1I+b_1uu^T$ and $A_2=a_2I+b_2uu^T,$
then
\begin{equation}
\det(\wedge^2 A_1+\wedge^2 A_2)
=
\bigl(a_1(a_1+b_1)+a_2(a_2+b_2)\bigr)^{n-1}
\cdot \bigl(a_1^2+a_2^2\bigr)^{\binom{n-1}{2}}.
\end{equation}
\end{enumerate}
\end{lemma}

\begin{proof}
Since \(u\) is a unit vector, 
\(Au=(aI+buu^T)u=(a+b)u.\)
On the other hand, if \(v\in u^\perp\), then \(u^T v=0\), and hence
\(
Av=(aI+buu^T)v=av.
\)
Thus \(A\) has eigenvalue \(a+b\) in the direction of \(u\) with multiplicity \(1\), and eigenvalue \(a\) on \(u^\perp\) with multiplicity \(n-1\). 
This implies the formula for $\det A,$ and the formula for $\det \wedge^2 A$ follows from a general formula $\det \wedge^2 A = \pr{\det A}^{n-1}.$

Next, (ii) follows from (i) and a direct computation using $A^{-1}A=I$.

Finally, the decomposition
$\wedge^2\mathbb{R}^n
=(u\wedge u^\perp)\oplus \wedge^2 u^\perp$
simultaneously diagonalizes \(\wedge^2 A_1\) and \(\wedge^2 A_2\). 
On \(u\wedge u^\perp\), the eigenvalue of \(\wedge^2 A_j\) is \(a_j(a_j+b_j)\), while on \(\wedge^2 u^\perp\), the eigenvalue is \(a_j^2\). 
Hence \(\wedge^2 A_1+\wedge^2 A_2\) has an eigenvalue
$a_1(a_1+b_1)+a_2(a_2+b_2)$
on \(u\wedge u^\perp\) with multiplicity \(n-1\), and an eigenvalue
$a_1^2+a_2^2$
on \(\wedge^2 u^\perp\) with multiplicity \(\binom{n-1}{2}\). Therefore
\[
\det\!\left(\wedge^2 A_1+\wedge^2 A_2\right)
=
\left(a_1(a_1+b_1)+a_2(a_2+b_2)\right)^{n-1}
\left(a_1^2+a_2^2\right)^{\binom{n-1}{2}}.
\]
The proof is complete.
\end{proof}

\begin{proof}
[Proof of Proposition~\ref{prop:AI-sym}]
	We first prove \eqref{AI-sym-2}.
	We will evaluate $\Lambda$ in terms of $\rho$ and $\phi,$ and then show that if $\rho'=0$ at a point, then $\Lambda=0$ at the point.
	This then proves~\eqref{AI-sym-2} based the positivity of $\Lambda$ by Definition~\ref{def:AI}.

We may choose $u_j$'s such that the matrix
	\begin{align*}
	U = \pr{u\,\,\bd_2u\,\,\cdots\,\,\bd_n u} \in{SO}(n)
	\end{align*} 
	if we view $u=(u_1,\cdots,u_n)^T$ as a column vector. 
    In particular, $UU^T$ is the identity.

	To evaluate $\Lambda,$ we use the tangent vectors $F_r$ and $F_j$'s ($j=2,\cdots,n$) to form the two $n\times n$ matrices 
	\begin{align*}
	X=\begin{pmatrix}
	 x' u
	& x\bd_2u
	& \cdots
	&  x\bd_nu\\
	\end{pmatrix}\,
    \text{ and }\,
	Y=\begin{pmatrix}
	y' u
	& y\bd_2u
	& \cdots
	& y\bd_nu\\
	\end{pmatrix}.
	\end{align*}
	Moreover, when checking the almost isoclinic condition, we can look at $\td X:=XU^T$ and $\td Y:=YU^T$ since $U^T\in{SO}(n).$
    Using $X=xU+(x'-x)\begin{pmatrix}u& 0& \cdots &0\end{pmatrix}$, we derive that \begin{equation}\label{XUT-formula}
	\begin{split}
	\td X
	= x I
	+ (x' - x) uu^T
	\text{ and  }
	\td Y
	= yI
	+ (y'-y) uu^T.
	\end{split}
	\end{equation}
Therefore, by Lemma~\ref{aI+buu^T},
\[\det\!\left(\wedge^2 \td X+\wedge^2 \td Y\right) =(xx'+yy')^{n-1}(x^2+y^2)^{\binom{n-1}{2}}=(\rho\rho')^{n-1}(\rho^2)^{\binom{n-1}{2}}\]
and
\[
\det\!\left(\td X^T \td X+\td Y^T \td Y\right)
=
\bigl((x')^2+(y')^2\bigr)\bigl(x^2+y^2\bigr)^{n-1}.
\]
Since $x+iy=\rho e^{i\phi}$, $x'+iy'=(\rho'+i\rho\phi')e^{i\phi}$ and $xx'+yy'=\rho\rho'$.
Combining these with Lemma~\ref{Lambda_general} proves \eqref{Lambda-formula-sym}.
	
	Using the formula of $\Lambda,$ \eqref{Lambda-formula-sym}, we now prove \eqref{AI-sym-2} in the proposition.
	If $\rho'=0$ at a point of the immersion, then we get $\Lambda=0.$
	Thus, the immersion is not almost isoclinic.
	This proves \eqref{AI-sym-2}.

	Next, we prove \eqref{AI-sym-1}.
	From above, we know that $\Lambda>0$ if $\rho'>0.$
	Thus, it remains to check the sign of $1+\lambda_j\lambda_k$.
    By Remark~\ref{rmk:general-matrices-for-AI}, it suffices to work with $\td X$ and $\td Y.$
	Given any angle $\varphi\in\bb R,$ we have
	\begin{align*}
	\td X_\varphi
	& = \cos\varphi \cdot \td X 
	+ \sin\varphi\cdot \td Y
	= \pr{x\cos\varphi + y\sin\varphi}
	+ ((x'-x)\cos\varphi+(y'-y)\sin\varphi)uu^T
	\text{ and}\\
	\td Y_\varphi
	& = -\sin\varphi\cdot \td X
	+ \cos\varphi\cdot \td Y
	= \pr{-x\sin\varphi + y\cos\varphi}
	+ (-(x'-x)\sin\varphi+(y'-y)\cos\varphi)uu^T.
	\end{align*}
	Applying (i) of Lemma \ref{aI+buu^T} to $\td X_\varphi$, we obtain
	\begin{align}\label{eq:detX}
	\det \td X_\varphi
	& =  \pr{x'\cos\varphi
		+ y'\sin\varphi}
	\cdot \pr{x\cos\varphi+y\sin\varphi}^{n-1}.
	\end{align}
	Thus, at a point $(r,u),$ we choose $\varphi=\phi(u),$ which implies
	\begin{align*}
	x'\cos\varphi + y'\sin\varphi
	& = \rho'\cos(\varphi-\phi)
	+ \rho\phi'\sin(\varphi-\phi)
	= \rho'
	\text{ and}\\
	x\cos\varphi + y\sin\varphi
	& = \rho\cos(\varphi-\phi)
	= \rho.
	\end{align*}
	Such a choice of $\varphi$ implies $\det \td X_\varphi>0$ at the point $(r,y).$
	Moreover,
	\begin{align*}
	\td X_\varphi
	= \rho I+(\rho' - \rho) uu^T
	\text{ and }
	\td Y_\varphi
	 = {\rho\phi'} uu^T.
	\end{align*}
	Applying formula (ii) of Lemma \ref{aI+buu^T} to $\td X_\varphi^{-1}$, we see that 
	$\td Y_\varphi \td X_\varphi^{-1}$ is of rank one.
	Hence, the only possibility of $1+\lambda_j\lambda_k$ for any $j<k$ is $1,$ so $F$ is almost isoclinic.
\end{proof}

Proposition~\ref{prop:AI-sym} characterizes the almost isoclinicity among ${SO}(n)$-symmetric Lagrangians in a direct way.
Based on that, we obtain an additional embeddedness property for ${SO}(n)$-symmetric Lagrangians.

\begin{cor}
	\label{cor:SOn-emb}
	Let $F\colon I\times \mathbb S^{n-1}\to\bb R^{2n}$ be an ${SO}(n)$-symmetric Lagrangian defined by functions $\rho$ and $\phi$ on an interval $I.$
	If $\rho>0$ and $\rho'>0,$ then $F$ defines an embedding.
	In particular, an almost isoclinic submanifold defined in this way is always embedded.
\end{cor}

\begin{proof}
	Recall that $F$ is defined in \eqref{F-sym-def} with $x+iy=\rho e^{i\phi}.$
	Suppose there exist $(r,u),(s,v)\in I\times \mathbb S^{n-1}$ such that
	\begin{align}\label{F-embedding}
	F(r,u)
	= F(s,v).
	\end{align}
	We will show that $(r,u)=(s,v),$ and hence the embeddedness of $F$ follows.
	
	Since $\rho>0,$ we know that $x(r)\neq 0$ or $y(r)\neq 0.$
	We assume $x(r)\neq 0$ and the case when $y(r)\neq 0$ is similar.
	The condition~\eqref{F-embedding} then implies $x(s)\neq 0,$ and more precisely,
	\begin{align*}
	x(r) u_k = x(s) v_k
	\end{align*}
	for all $k=1,\cdots,n.$
	If $u_k\neq 0,$ this then implies that the quantity
	$\frac{u_k}{v_k}=\frac{x(s)}{x(r)}$
	is independent of $k.$
	Thus, the only possibilities are $u=v$ or $u=-v.$
	
	If $u=v,$ we then get $x(r)=x(s),$ and similarly $y(r)=y(s).$
	This implies $\rho(r)=\rho(s).$
	Since $\rho'>0,$ we then conclude $r=s.$
	
	If $u=-v,$ then $x(r)=-x(s)$ and similarly $y(r)=-y(s).$
	This again implies $\rho(r)=\rho(s),$ so by $\rho'>0,$ we conclude $r=s,$ which contradicts $u=-v.$
	
	Thus, we prove that \eqref{F-embedding} implies $r=s$ and $u=v.$
	This proves that $F$ defines an embedding.
\end{proof}

Next, we use Proposition~\ref{prop:AI-sym} to construct properly embedded Lagrangians that are almost isoclinic.

\begin{prop}
	\label{prop:SOn-general-example}
	Given $0<a<b<\infty$ and $C^1$ functions $\rho,\phi\colon[a,b]\to\bb R$ such that $\rho>0$ and $\rho'>0,$ there exists a properly embedded Lagrangian that is almost isoclinic such that it contains the submanifold defined by
	\begin{align*}
	F(r,u)
	= F(r,u_1,\cdots,u_n)
	= \pr{x(r)\,u_1,
		\cdots ,
		x(r)\,u_n,
		y(r)\,u_1,
		\cdots,
		y(r)\,u_n}
	\end{align*}
	for $r\in [a,b],$ $u=(u_1,\cdots,u_n)\in \mathbb S^{n-1}$ as a subset, where $x$ and $y$ are real-valued functions such that
	$x+iy = \rho e^{i\phi}.$
\end{prop}

\begin{proof}
	We extend $\rho$ and $\phi$ to functions $\td\rho$ and $\td\phi$ on $[0,\infty)$ such that
	\begin{enumerate}
		\item 
		$\td\rho>0$ and $\pr{\td\rho}'>0$ on $(0,\infty),$
		\item 
		there exists $\varepsilon>0$ such that $\td\rho(r)=r$ and $\td\phi(r)=0$ for $r\in[0,\varepsilon],$ and
		\item 
		$\td\phi$ is constant on $[\varepsilon^{-1},\infty).$
	\end{enumerate}
	We can then use $\td\rho$ and $\td\phi$ to construct a new embedding 
	\begin{align*}
	\td F(r,u)
	= \begin{cases}
	\pr{\td x(r)\,u_1,
		\cdots ,
		\td x(r)\,u_n,
		\td y(r)\,u_1,
		\cdots,
		\td y(r)\,u_n}
	&\text{if }r>0\\
	0&\text{if }r=0
	\end{cases},
	\end{align*}
	where $\td x$ and $\td y$ are real-valued functions such that 
	$\td x+\td iy = \td\rho e^{i\td\phi}.$
	Note that on $(\varepsilon,\infty)\times \mathbb S^{n-1},$ $\td F$ defines an almost isoclinic embedding by Proposition~\ref{prop:AI-sym} and Corollary~\ref{cor:SOn-emb}.
	On the other hand, on $[0,\varepsilon)\times \mathbb S^{n-1},$ $\td F$ is a flat disc, which is automatically Lagrangian and almost isoclinic.
	The strict monotonicity of $\td \rho$ guarantees that these two pieces do not intersect, so the whole $\td F$ is still an embedding.
	
	Finally, our choice of $\td\phi$ at infinity tells us that the submanifold is asymptotically flat.
	In particular, it is complete.
\end{proof}

As a corollary, we obtain examples of almost isoclinic Lagrangians that are not graphical over any rotated $n$-planes.

\begin{cor}
	\label{cor:non-graphical-ex}
	There exist almost isoclinic Lagrangian surfaces in $\bb C^n$ that are complete and non-graphical over $e^{i\varphi}\bb R^n$ for any $\varphi\in\bb R.$
\end{cor}

\begin{proof}
	We can just take $F$ of the form \eqref{F-sym-def} for $x(r), y(r)$ to be determined. For any $\varphi\in\bb R,$ to check the graphicality of $F$ over $e^{i\varphi}\bb R^n,$ as in \eqref{eq:detX}, Lemma~\ref{aI+buu^T} implies
	\begin{align*}
	\det \pr{\cos\varphi \cdot \td X + \sin\varphi\cdot \td Y}
    =\pr{x'\cos\varphi+ y'\sin\varphi} \cdot \pr{x\cos\varphi+y\sin\varphi}^{n-1}.
	\end{align*}
    Consider a cutoff function $\eta\colon[0,\infty)\to\bb R_{\ge 0}$ such that $\eta$ is non-decreasing and
	\begin{align*}
	\eta(r)=\begin{cases}
	1&\text{for }r\ge 2\\
	0&\text{for }r\le 1
	\end{cases}.
	\end{align*}
	Using this, we take
	\begin{align*}
	x(r):=r \cos (\eta(r)r)\text{ and }
	y(r):=r\sin (\eta(r) r). 
	\end{align*}
	These functions satisfy all the properties required in the proof of Proposition~\ref{prop:SOn-general-example}.
	As a result, they define an almost isoclinic ${SO}(n)$-symmetric Lagrangian embedding $F$ as in \eqref{F-sym-def}. For $r\geq 2$, \begin{align*}
	\det \pr{\cos\varphi \cdot \td X + \sin\varphi\cdot \td Y}
    =\pr{x'\cos\varphi+ y'\sin\varphi} \cdot r^{n-1} \pr{\cos(r-\varphi)}^{n-1}.
	\end{align*}
	Hence, for any $\varphi\in\bb R,$ there exists $r_\varphi\ge 2$ such that $r_\varphi-\varphi-\pi/2$ is a multiple of $2\pi.$ 
	For this $r_\varphi,$ we get
	\begin{align*}
	\det \pr{\cos\varphi \cdot X + \sin\varphi\cdot Y}(r_\varphi,\theta)=0.
	\end{align*}
	Thus, $F$ is not graphical over $e^{i\varphi}\bb R^n$ at $\pr{r_\varphi,\theta}.$
\end{proof}

\begin{figure}[h]
	\centering
	\includegraphics[width=6cm]{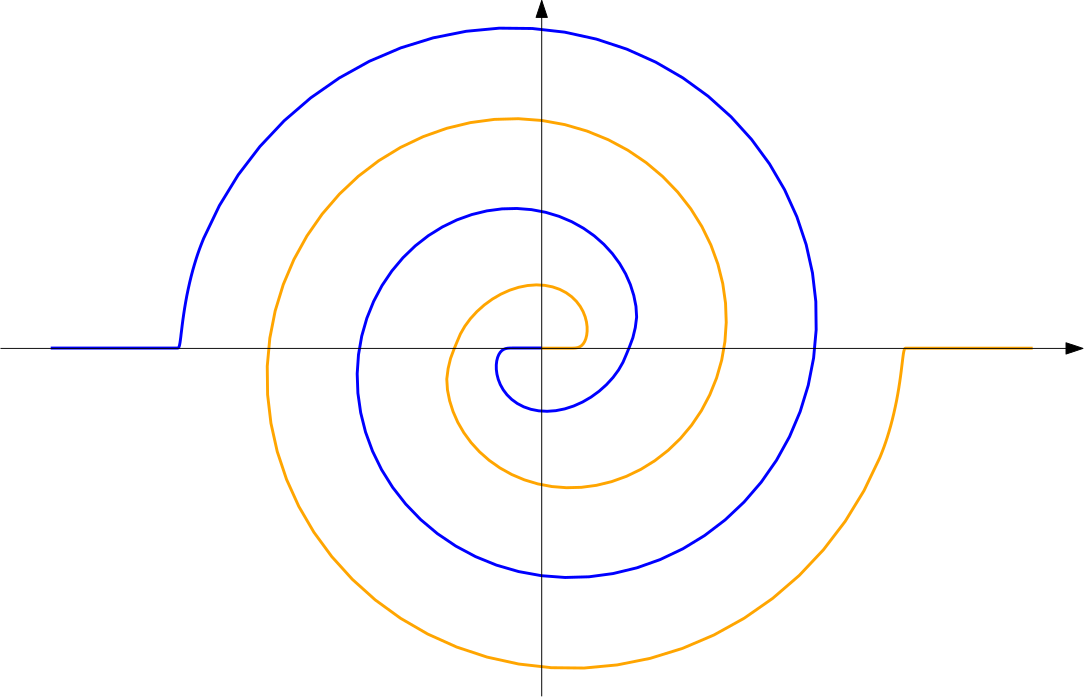}
	\caption{
    The orange curve is an example constructed following Proposition~\ref{prop:SOn-general-example}.
    The function $\td \rho(r)$ is simply $r$ and a small positive number $\varepsilon$ is chosen such that 
    $\td\phi(r)=r$ for $r\in [\varepsilon, 4\pi-\varepsilon]$ and $\td\phi|_{[0,\infty)\setminus[\varepsilon/2, 4\pi-\varepsilon/2]}=0.$
    A direct calculation shows that for this example, the projection of the Gauss map onto the self-dual spherical part in onto to the equator indicated in Figure~\ref{fig:Gauss}.
    }
	\label{fig:spiral}
\end{figure}

As in Figure~\ref{fig:spiral}, the examples found in Corollary~\ref{cor:non-graphical-ex} are diffeomorphic to $\bb R^n.$
It will be interesting to ask for properly embedded complete almost isoclinic Lagrangians with other topological types.

\subsection{Lawlor necks}
\label{sec:Lawlor-neck}

In this section, we study the almost isoclinicity property for {\bf Lawlor necks}, a special class of asymptotically conical special Lagrangians constructed by Lawlor~\cite{L89}.
We will show that none of them is almost isoclinic.

Let $n\ge 2$ and $a_1,\cdots,a_n$ be positive numbers.
Consider the polynomial
\begin{align*}
P(y)
:= \frac 1{y^2}\pr{\prod_{k=1}^n(1+a_ky^2)-1}.
\end{align*}
Using this polynomial, we define
\begin{align*}
\phi_k(r)
&:= \int_{-\infty}^r \frac{a_k}{\pr{1+a_ky^2}\sqrt{P(y)}}dy\,
\text{ and }\,
\rho_k(r)
:= \sqrt{\frac 1{a_k} + r^2}
\end{align*}
for $r\in\bb R.$
These data lead to a curve $Z=\set{\pr{z_1(r),\cdots,z_n(r)}:r\in\bb R}$ in $\bb C^n\simeq \bb R^{2n}$ where 
$$z_k(r) := \rho_k(r) e^{i\phi_k(r)}.$$
Rotating it, we get a map $F\colon \bb R\times \mathbb S^{n-1}\to\bb R^{2n}$ defined by
\begin{align*}
F(r,u)
= F(r,u_1,\cdots,u_n)
= \pr{z_1(r)u_1,\cdots,z_n(r)u_n}
\in \bb C^n\simeq \bb R^{2n}.
\end{align*}
The image of this immersion
\begin{align*}
L_a
:=\set{\pr{z_1(r)u_1,\cdots,z_n(r)u_n}: r\in\bb R\text{ and }(u_1,\cdots,u_n)\in \mathbb S^{n-1}}
\end{align*} 
is called the Lawlor neck associated to the parameter $a=(a_1,\cdots,a_n).$

\begin{prop}
	\label{prop:L-not-AI-n>2}
	Let $n\ge 2$ and $F\colon \bb R\times \mathbb S^{n-1}\to L_a$ be the Lawlor neck with parameter $a=(a_1,\cdots,a_n).$
	If we let $\Lambda_a$ be the function $\Lambda$ defined on $L_a,$ then there exists $(r,u)\in \bb R\times \mathbb S^{n-1}$ such that $\Lambda_a(r,u)=0.$
	In particular, $L_a$ is not almost isoclinic.
\end{prop}

When the parameters $a_k$'s are all the same, the associated Lawlor neck is ${SO}(n)$-symmetric in the sense of \eqref{F-sym-def}.
This special class of Lawlor necks are called {\bf Lagrangian catenoids} and were first constructed in~\cite{HL}.
From Proposition~\ref{prop:AI-sym}, it is straightforward to check that such Lagrangian catenoids are not almost isoclinic since the defining function $\rho$ is not monotone.
Proposition~\ref{prop:L-not-AI-n>2} says that the whole family of Lawlor necks are also not almost isoclinic.

\begin{proof}
	Suppose we are given $a=(a_1,\cdots,a_n).$
	We will show that the function $\Lambda=\Lambda_a$ on $L_a$ has a zero, and hence $L_a$ is not almost isoclinic.
	
	Consider the following parametrization
	\begin{align*}
	F(r,u)
	= F(r,u_1,\cdots,u_n)
	= \pr{x_1\,u_1,
		\cdots ,
		x_n\,u_n,
		y_1\,u_1,
		\cdots,
		y_n\,u_n}
	\end{align*}
	where $x_k=\rho_k\cos\phi_k,$ $y_k=\rho_k\sin\phi_k,$ and $u=(u_1,\cdots,u_n)\in \mathbb S^{n-1}.$
	We assume that the embedding $u\colon \mathbb S^{n-1}\to\bb R^n$ is parametrized by $(n-1)$ variables $(\theta_2,\cdots,\theta_n)$ such that when $u=(1,0,\cdots,0),$
	\begin{align*}
	\bd_ju:=\bd_{\theta_j}u=e_j
	\end{align*}
	where $e_j$'s form the standard basis of $\bb R^n.$
	At such a point, the tangent space is spanned by $\begin{pmatrix}X\\Y
	    \end{pmatrix}$ where 
	\begin{align*}
	X={\rm diag}\pr{x_1', x_2,\cdots,x_n}
	\text{ and  }
	Y={\rm diag}\pr{y_1', y_2,\cdots,y_n}
	\end{align*}
	It then follows that $X^TX+Y^TY={\rm diag}\pr{(x_1')^2+(y_1')^2, (x_2)^2+(y_2)^2,\cdots,(x_n)^2+(y_n)^2}$ at $(r,u)=(r,1,0,\cdots,0)$ for any $r\in\bb R$. We can therefore normalize to $\begin{pmatrix}\td X\\\td Y
	    \end{pmatrix}$ such that $\td X+i\td Y\in U(n)$.
	
	We start evaluating $\Lambda$ at $(r,u)=(r,1,0,\cdots,0)$ for any $r\in\bb R.$
	At such a point, both $\wedge^2 \td X$ and $\wedge^2\td Y$ are diagonal and their sum is 
	\begin{align*}
	\wedge^2 \td X + \wedge^2 \td Y
	=&{\rm diag}\pr{
		\pr{\frac{x_1'x_\ell + y_1'y_\ell}
			{\sqrt{\pr{x_1'}^2 + \pr{y_1'}^2}\sqrt{x_\ell^2+y_\ell^2}}
		}_{\ell\ge 2},
		\pr{\frac{x_jx_k + y_jy_k}
			{\sqrt{x_j^2+y_j^2}\sqrt{x_k^2+y_k^2}}}_{2\le j<k<n}
	}.
	\end{align*}
    We note that the first expression on the right hand side is exactly the cosine of the angle between the two 2-vectors $(x_1', y_1')$ and $(x_\ell, y_\ell)$, while the second expression is the cosine of the angle between the two 2-vectors $(x_j, y_j)$ and $(x_k, y_k)$.
	Recall that $x_k+iy_k=\rho_ke^{i\phi_k}$.  
    Noting that
	\begin{align*}
	\alpha:=\rho_k\rho_k'=r\text{ and }
	\beta:=\rho_k^2\phi_k'
	= \frac 1{\sqrt P}
	\end{align*}
	are independent of $k$, we have 
    \[x'_k+iy'_k=\frac{1}{\rho_k}(\alpha+i\beta)e^{i\phi_k}.\]
	These imply
	\begin{align*}
	\Xi_\ell
    :=\frac{x_1'x_\ell + y_1'y_\ell}
	{\sqrt{\pr{x_1'}^2 + \pr{y_1'}^2}\sqrt{x_\ell^2+y_\ell^2}}
	& = 	\frac\alpha{\sqrt{\alpha^2+\beta^2}} \cos(\phi_\ell-\phi_1) + \frac\beta{\sqrt{\alpha^2+\beta^2}} \sin(\phi_\ell-\phi_1)
	\end{align*}
	for $\ell\ge 2$ and
	\begin{align*}
	\frac{x_jx_k + y_jy_k}
	{\sqrt{x_j^2+y_j^2}\sqrt{x_k^2+y_k^2}}
	& = \cos\pr{\phi_k-\phi_j}
	\end{align*}
	for $2\le j<k.$
	This information completely determines $\Lambda$ at $(r,u)=(r,1,0,\cdots,0).$
	
	Finally, we will show that $\Lambda=\det\pr{\wedge^2X + \wedge^2Y}$ has a zero.
	Note that for any $j$ and $k,$ by the definition of $\phi_j,$
	\begin{align*}
	\lim_{r\to-\infty}\pr{\phi_j(r)-\phi_k(r)} 
	= \lim_{r\to-\infty}\phi_j(r)
	- \lim_{r\to-\infty}\phi_k(r)
	=0-0=0.
	\end{align*} 
	Thus, if we fix some $\ell=2,\cdots,n,$ for $r$ close enough to $-\infty$,
	$\Xi_\ell$ will be negative since $\alpha(r)\to-\infty$ and $\beta(r) = \frac 1{\sqrt{P(r)}}\to 0$ as $r\to-\infty.$
	On the other hand, when $r=0,$ we have $\Xi_\ell\ge 0$
	if we assume $\phi_1$ is the smallest angle among $\phi_k$'s.
	Thus, there exists $r_\ell \le 0$ such that $\Xi_\ell(r_\ell,1,0,\cdots,0)=0.$
	Therefore, we get
	\begin{align*}
	\Lambda(r_\ell,1,0,\cdots,0)
	= \prod_{\ell=2}^n 
	\frac{x_1'x_\ell + y_1'y_\ell}
	{\sqrt{\pr{x_1'}^2 + \pr{y_1'}^2}\sqrt{x_\ell^2+y_\ell^2}}
	\cdot \prod_{2\le j<k}
	\frac{x_jx_k + y_jy_k}
	{\sqrt{x_j^2+y_j^2}\sqrt{x_k^2+y_k^2}}
	= 0.
	\end{align*}
	In particular, this proves that $L_a$ is not almost isoclinic.
\end{proof}

Although Lawlor necks are not globally almost isoclinic, changing the reference base $n$-plane will allow one to identify almost isoclinic ends on them.

\section{\bf Proof of the log-concavity property}
\label{sec:proof-Lambda}

In this section, we will prove Proposition~\ref{prop:Hess-n-dim}.
If not specified, every sum $\sum\limits_j$ is over $j=1,\cdots,n.$
Every sum like $\sum\limits_{j,k}$ or $\sum\limits_{j<k}$ is under the same restriction, that is, for $1\le j,k\le n$ or $1\le j<k\le n$ respectively.

In the rest of the section, we will assume that in a graph chart based on $L_0\in{\rm Lag}^+(n),$ an $n$-plane $L_S$ is given by the graph of a symmetric matrix $S$ after we fix a basis of $L_0.$
We assume $L_S\in\mathcal{AI}(n)$ and $S={\rm diag}(\lambda_1,\cdots,\lambda_n)$ in the chosen basis, with respect to which the space $T_{L_S}{\rm Lag}^+(n)$ is isomorphic to the space of $n\times n$  symmetric matrices.
We always use the matrix representation $A=[\alpha]_{L_0}$ to express an element $\alpha\in T_{L_S}{\rm Lag}^+(n)$ in this section unless otherwise specified.

We start from simple properties of the Grassmannian metric.

\begin{lemma} 
	\label{lem:basic-prop}
	Suppose $L_S\in{\rm Lag}^+(n)$ is represented by an $n\times n$ matrix $S$ and suppose $A, B$ are symmetric matrices in the tangent space $T_S{\rm Lag}^+(n).$
	Denote
	$G = (I + S^2)^{-1}.$
	Then the Grassmannian metric is
	\begin{align}\label{metric-on-Lag}
	\langle A, B \rangle = \operatorname{tr}(G A G B),
	\end{align}
	and the Levi-Civita connection of the metric is
	\begin{align}\label{nabla-on-Lag}
	\nabla_A B = - A S G B - B S G A.
	\end{align}
	Moreover, the Hessian of a function $F\colon {\rm Lag}^+(n)\to\bb R$ at $S$ is
	\begin{align}\label{Hess-on-Lag}
	\Hess F(S) (A, B) = D_A(D_B(F))(S) - 
	D_{\nabla_A B}(F)(S),
	\end{align}
	where $D_A(F)(S)=\bd_t F(S+tA) |_{t=0}$ is the directional derivative. 
\end{lemma}

\begin{proof}
	First, we look at the metric \eqref{metric-on-Lag}. 
	As mentioned above, $A$ and $B$ represent two elements $\alpha$ and $\beta$ in $T_{L_S}{\rm Lag}^+(n)$ by one of the matrix representations $[\alpha]_{L_0}=A$ and $[\beta]_{L_0}=B.$
	As we are evaluating the metric on $T_{L_S}{\rm Lag}^+(n),$ when applying the Grassmannian metric~\eqref{Gr-metric}, we need to use $[\alpha]_{L_S}$ and $[\beta]_{L_S}.$
	To this end, we apply formula~\eqref{matrix-in-graph-chart} to get
	\begin{align*}
	\pair{A,B}_{\rm Gr}
	:= \pair{\alpha,\beta}_{g_S}
	= \tr \pr{[\alpha]_{L_S}^T [\beta]_{L_S}}
	= \tr \pr{G^{1/2}AG^{1/2} \cdot G^{1/2}BG^{1/2}}
	= \tr\, GAGB.
	\end{align*}
	This proves~\eqref{metric-on-Lag}.
	From now on, we will write $\pair{A,B}_{\rm Gr}
	= \pair{A,B}$ when it is clear.

	Next, we calculate the Levi-Civita connection \eqref{nabla-on-Lag}.
	The Levi-Civita connection $\n$ satisfies
	\begin{align}\label{connection-def}
	2\, \pair{\n_AB, C}
	= \pr{
		D_A\pair{B,C}
		+ D_B\pair{C,A}
		- D_C\pair{A,B}
	}
	\end{align}
	for any symmetric matrices $A,B,$ and $C$ representing three tangent vectors in $T_{L_S}{\rm Lag}^+(n).$
	Since $G\pr{1+S^2}=1,$ differentiation in $A$ implies
	$D_AG\pr{1+S^2}
	+ G\pr{SA+AS}
	= 0,$
	so
	\begin{align}\label{DG}
	D_AG
	= - G\pr{SA+AS} \pr{1+S^2}^{-1}
	= - G\pr{SA+AS} G.
	\end{align}
	As a consequence,
	\begin{align*}
	D_A\pair{B,C}
	= D_A\pr{\tr\, GBGC}
	& = \tr \pr{D_AG}BGC
	+ \tr\, GB\pr{D_AG}C\\
	& = - \tr\, GSAGBGC
	- \tr\, GASGBGC
	- \tr\, GBGSAGC
	- \tr\, GBGASGC.
	\end{align*}
	Similarly, we have
	\begin{align*}
	D_B\pair{C,A}
	&= - \tr\, GSBGCGA
	- \tr\, GBSGCGA 
	- \tr\, GCGSBGA
	- \tr\, GCGBSGA 
	\text{ and}\\
	D_C\pair{A,B}
	&= - \tr\, GSCGAGB
	- \tr\, GCSGAGB 
	- \tr\, GAGSCGB
	- \tr\, GAGCSGB.
	\end{align*}
	Thus, \eqref{connection-def} implies
	\begin{align*}
	2\tr \pr{G\,\n_AB\,GC}
	= & -\tr(GSAGBGC + GASGBGC+GBGSAGC+GBGASGC)\\
	& - \tr(GSBGCGA+GBSGCGA+GCGSBGA+GCGBSGA)\\
	& + \tr(GSCGAGB+GCSGAGB+GAGSCGB+GAGCSGB)\\
	= & -\tr(\udl{GSAGBGC} + GASGBGC+GBGSAGC+\udl{GBGASGC})\\
	& - \tr(\udl{GSBGAGC}+GBSGAGC+GAGSBGC+\udl{GAGBSGC})\\
	& + \tr(\udl{GAGBSGC+GSAGBGC+GBGASGC+GSBGAGC})\\
	= & -\tr \pr{G\pr{ASGB + BGSA + BSGA + AGSB}GC},
	\end{align*}
	where we freely use the relation $GS=SG$ and the formula $\tr XY=\tr YX$ to notice that the eight underlined terms cancel.
	Since the equality above holds for any $C,$ we finally get
	\begin{align*}
	\n_A B 
	= -\frac 12(ASGB + BGSA + BSGA + AGSB)
	= - ASGB - BGSA
	\end{align*}
	and hence \eqref{nabla-on-Lag} follows.
	Finally, \eqref{Hess-on-Lag} follows from definition.
\end{proof}

With the basic rules in Lemma~\ref{lem:basic-prop}, we can directly evaluate the Hessian of a function defined on ${\rm Lag}^+(n).$
In the following, we do this for the two functions involved in $\Lambda.$

\begin{lemma}
	\label{lem:first-expansion}
	Suppose $L_S\in{\rm Lag}^+(n)$ is represented by an $n\times n$ matrix $S$ and suppose $A$ is a symmetric matrix in the tangent space $T_{L_S}{\rm Lag}^+(n).$ 
	Denote
	$G = (I + S^2)^{-1}
	\text{ and }
	M = (I + \wedge^2 S)^{-1}.$
	Then the Hessians of the functions $\log \det (I + S^2)$ and $\log \det (I + \wedge^2 S)$ on $\mathcal{AI}(n)$ with respect to the Grassmannian metric are
	\begin{align}\label{Hess-Omega}
	\Hess \pr{\log \det (I + S^2)} (A, A)
	= 2 \operatorname{tr}(G A G A) + 2 \operatorname{tr}(G S A G S A)
	\end{align}
	and
	\begin{equation}\label{Hess-Lambda-first}
	\begin{split}
	\Hess \pr{\log \det (I+\wedge^2 S}(A,A)
	= & \, 2\tr \pr{M (A \wedge A) }\\
	& - \tr\pr{M (A \wedge S + S \wedge A) \, M (A \wedge S + S \wedge A) }\\
	& + 2\tr\pr{M \pr{S\wedge ASGA + ASGA \wedge S}}.
	\end{split}
	\end{equation}
\end{lemma}

\begin{proof}
	For simplicity, we write
	$F^\Omega(S)
	:= \log \det\pr{1+S^2}
	\text{ and }
	F^\Lambda(S):=\log \det\pr{1+{\wedge^2 S}}.$
	
	First, we prove \eqref{Hess-Omega}.
	Through a variation $S+tA,$ we get
	\begin{align*}
	D_AF^\Omega(S)
	= \bd_t \log \det(1+ S^2+tSA+tAS+t^2A^2)\big|_{t=0}
	& = \tr((1+S^2)^{-1}(SA+AS))
	= 2 \tr((1+S^2)^{-1}SA).
	\end{align*}
	Thus, using \eqref{nabla-on-Lag} and \eqref{DG}, we obtain
	\begin{align*}
	\Hess{F^\Omega}(S)(A,A) 
	= & D_A\pr{ 2 \tr (1+S^2)^{-1}SA}
	- D_{\nabla_A A} F^\Omega(S)\\
	= & 2\pr{
		\tr (1 + S^2)^{-1} A A
		- \tr(1 + S^2)^{-1} (S A + A S) (1 + S^2)^{-1} S A}	 \\
	& +  2\tr((1+S^2)^{-1}S
	(ASGA + AGSA)
	)\\
	= & 2\pr{
		\tr\,GAA-\tr\,G(SA+AS)GSA
	} +  4\,\tr\,GSAGSA\\
	= & 2\,\tr\, GA^2 
	- 2\,\tr\, GASGSA
	+ 2\,\tr\,GSAGSA\\
	= & 2\,\tr\, GA\pr{I-S^2G}A + 2\,\tr\,GSBGSA\\
	= & 2\,\tr\, GAGA + 2\,\tr\,GSBGSA
	\end{align*}
	where we use $1-S^2G=G$ in the last equality.
	
	Nest, we prove \eqref{Hess-Lambda-first}.
	Through a variation $S+tA,$ we get
	\begin{align*}
	D_AF^\Lambda(S)
	= \bd_t \log\det\pr{1+\wedge^2S+tS\wedge A + tA\wedge S + t^2 A\wedge A}|_{t=0}
	& = \tr \pr{\pr{1+\wedge^2S}^{-1} (S\wedge A+A\wedge S)}.
	\end{align*}
	Using \eqref{nabla-on-Lag} and \eqref{DG} again, we derive
	\begin{align*}
	\Hess{F^\Lambda}(S)(A,A)
	= & D_A\pr{dF^\Lambda_S(A)} - D_{\nabla_A A} F^\Lambda(S) \\
	= & \tr \pr{ M (A \wedge A + A \wedge A) 
		- M (A \wedge S + S \wedge A) \, M (A \wedge S + S \wedge A) }\\
	& + \tr\pr{M \pr{S\wedge (ASGA + AGSA) + (ASGA + AGSA)\wedge S}}.
	\end{align*}
	This is the same as \eqref{Hess-Lambda-first} and hence finishes the proof of the lemma.
\end{proof}

The function $\Lambda$ at $S$ satisfies
\[\log \Lambda(S)
= \log \det (I+\wedge^2S)-\frac{n-1}{2}\log\det(I+S^2)
.\]
Therefore, from Lemma~\ref{lem:first-expansion}, we can write
\begin{align*}
\Hess\pr{\log \Lambda(S)}(A, A)=T_1+T_2+T_3+T_4
\end{align*} 
where 
\begin{equation}\label{T-1-T-4}
\begin{split}
T_1
&= 2\,\tr\pr{M (A\wedge A)},\\
T_2
&= - \tr \pr{M (A \wedge S + S \wedge A)\, M (A \wedge S + S \wedge A)},\\
T_3
&=2\,\tr\pr{M (S \wedge A S G A + A S G A \wedge S)},\text{ and}\\
T_4
&=-(n-1) \tr (GAGA) -(n-1) \tr (GSAGSA).
\end{split}
\end{equation}
Our strategy is to write each $T_k$ as $T_k'+ T_k''$ and then show $T_1'+T_2'+T_3'+T_4'\leq 0$ and $T_1''+T_2''+T_3''+T_4''\leq 0$ respectively.
The way we decompose $T_k$ to $T_k'+T_k''$ is to separate the terms to diagonal terms and off-diagonal terms in $\wedge^2\bb R^n.$
We will deal with each of them in the following four lemmas.

\begin{lemma}
	[$T_1$'s expansion]
	\label{lem:T1}
	For $A=(a_{jk})_{j,k},$
	\begin{align*}
	T_1
	= 2\,\tr\pr{M (A\wedge A)}
	= 2 \sum_{j<k} \frac 1{1+\lambda_j\lambda_k} 
	\pr{a_{jj}a_{kk} - a_{jk}^2}
	\end{align*}
	and we let $T_1'=T_1$ and $T_1''=0.$
\end{lemma}

\begin{proof}
	Remember that we are working with an orthonormal basis $\set{v_j}_{j=1,\cdots,n}$ which diagonalizes $S={\rm diag}(\lambda_1,\cdots,\lambda_n)$ and write $A=(a_{jk})_{j,k}$ in the same orthonormal basis.
	Note that $v_j\wedge v_k$'s form a basis of $\wedge^2\bb R^n$ for $j<k,$ and the $(jk,pq)$-entry of $A\wedge A$ is given by
	\begin{align*}
	\pair{(A\wedge A)(v_j\wedge v_k), v_p\wedge v_q}
	& = \ppair{\sum_c a_{jc}v_c\wedge \sum_d a_{kd}v_d, v_p\wedge v_q}
	= a_{jp}a_{kq} - a_{jq}a_{kp}.
	\end{align*}
	Thus, since $M$ is diagonal with entries $1/(1+\lambda_j\lambda_k)$'s, we get
	\begin{align*}
	2 \tr\pr{M(A\wedge A)}
	= 2 \sum_{j<k}
	\frac 1{1+\lambda_j\lambda_k} 
	\pr{a_{jj}a_{kk} - a_{jk}^2}
	\end{align*}
	so the formula follows.
\end{proof}

\begin{lemma}[$T_2$'s expansion]
	\label{lem:T2}
	For $A=(a_{jk})_{j,k},$
	\begin{align*}T_2
	&= - \operatorname{tr}\big(M (A \wedge S + S \wedge A)\, M (A \wedge S + S \wedge A)\big)\\
	&=- \sum_{j<k}\sum_{p<q}
	\frac 1{1+\lambda_j\lambda_k}
	\frac 1{1+\lambda_p\lambda_q}
	\pr{a_{jp}\lambda_k\delta_{kq} - a_{jq}\lambda_k\delta_{kp}
		+ \lambda_j\delta_{jp}a_{kq}
		- \lambda_j\delta_{jq}a_{kp}}^2\\
	&=T_2'+T_2'',
	\end{align*}
	where 
	\begin{align*}T_2'&=- \sum_{j<k}
	\frac 1{(1+\lambda_j\lambda_k)^2} 
	\pr{\lambda_k a_{jj}
		+ \lambda_j a_{kk}
	}^2\text{ and}\\
	T_2''
	&=-\sum_{j<k}\sum_{\substack{m\neq j,k}}
	\frac{2\lambda_m^2}
	{(1+\lambda_j\lambda_m)(1+\lambda_k\lambda_m)}
	\, a_{jk}^2.
	\end{align*}
\end{lemma}

\begin{proof}
	As in Lemma~\ref{lem:T1}, we work with an orthonormal basis $\set{v_j}_{j=1,\cdots,n}$ which diagonalizes $S={\rm diag}(\lambda_1,\cdots,\lambda_n)$ and write $A=(a_{jk})_{j,k}$ in the same orthonormal basis.
	Since $v_j\wedge v_k$'s form a basis of $\wedge^2\bb R^n$ for $j<k,$
	the $(jk,pq)$-entry of $A\wedge S+ S\wedge A$ is
	\begin{align*}
	C_{jkpq}
	:= \pair{(A\wedge S+ S\wedge A)(v_j\wedge v_k), v_p\wedge v_q}
	& = \ppair{
		\sum_c a_{jc}v_c\wedge \lambda_kv_k
		+ \lambda_jv_j\wedge \sum_d a_{kd}v_d,
		v_p\wedge v_q
	}\\
	& = a_{jp}\lambda_k\delta_{kq} - a_{jq}\lambda_k\delta_{kp}
	+ \lambda_j\delta_{jp}a_{kq}
	- \lambda_j\delta_{jq}a_{kp},
	\end{align*}
	and hence it is symmetric in $(jk,pq)$ in the sense that
	\begin{align*}
	C_{pqjk}
	& = a_{pj}\lambda_q\delta_{qk} - a_{pk}\lambda_q\delta_{qj}
	+ \lambda_p\delta_{pj}a_{qk}
	- \lambda_p\delta_{pk}a_{qj}
	= C_{jkpq}.
	\end{align*}
	Thus, if we let $m_{jk}:=1/(1+\lambda_j\lambda_k),$ then
	\begin{align*}
	\tr\pr{M (A \wedge S + S \wedge A) \, M (A \wedge S + S \wedge A) }
	= & \sum_{j<k, p<q}
	m_{jk} C_{jkpq} m_{pq} C_{pqjk}\\
	= & \sum_{j<k, p<q}
	m_{jk} m_{pq}
	\pr{a_{jp}\lambda_k\delta_{kq} - a_{jq}\lambda_k\delta_{kp}
		+ \lambda_j\delta_{jp}a_{kq}
		- \lambda_j\delta_{jq}a_{kp}}^2.
	\end{align*}
	This proves the first identity of the lemma.
	
	Next, we split $T_2$ into two terms $T_2'$ and $T_2''$ by
	\begin{align*}
	T_2'
	&:=- \sum_{\substack{j<k, p<q\\(j,k)= (p,q)}}
	\frac 1{1+\lambda_j\lambda_k}
	\frac 1{1+\lambda_p\lambda_q}
	\pr{a_{jp}\lambda_k\delta_{kq} - a_{jq}\lambda_k\delta_{kp}
		+ \lambda_j\delta_{jp}a_{kq}
		- \lambda_j\delta_{jq}a_{kp}}^2
	\text{ and}\\
	T_2''
	&:= -\sum_{\substack{j<k, p<q\\(j,k)\neq (p,q)}}
	\frac 1{1+\lambda_j\lambda_k}
	\frac 1{1+\lambda_p\lambda_q}
	\pr{a_{jp}\lambda_k\delta_{kq} - a_{jq}\lambda_k\delta_{kp}
		+ \lambda_j\delta_{jp}a_{kq}
		- \lambda_j\delta_{jq}a_{kp}}^2.
	\end{align*}
	The expression of $T_2'$ claimed in the lemma then follows directly by $\delta_{jj}=\delta_{kk}=1$ and $\delta_{jk}=0.$

	$T_2''$ can be rewritten as
	\begin{align*}
	T_2''
	=
	-\sum_{\substack{\ell<m,\ p<q\\(\ell,m)\neq (p,q)}}
	\frac 1{1+\lambda_\ell\lambda_m}
	\frac 1{1+\lambda_p\lambda_q}
	\pr{
		a_{\ell p}\lambda_m\delta_{mq}
		- a_{\ell q}\lambda_m\delta_{mp}
		+ \lambda_\ell\delta_{\ell p}a_{mq}
		- \lambda_\ell\delta_{\ell q}a_{mp}
	}^2.	
	\end{align*}
	Note that since $(\ell,m)\neq (p,q),$ the pairs will contribute only when $\set{\ell,m}\cap \set{p,q}$ has exactly one element.
	Thus, when we fix two numbers $j<k,$ we have the following four cases.
	We let $r$ be the common index, which will be different from $j$ or $k.$
	
	When $\ell=p=r,$ the third term in the parenthesis is the only non-zero term and we get the term containing $a_{jk}^2$ when $\set{m,q}=\set{j,k}.$
	This then leads to
	\begin{align}\label{T2-1}
	\sum_{\substack{r\neq j,k\\r<j,r<k}}
	\frac 1{1+\lambda_r\lambda_j}
	\frac 1{1+\lambda_r\lambda_k}
	\lambda_r^2 a_{jk}^2
	+ \sum_{\substack{r\neq j,k\\r<k,r<j}}
	\frac 1{1+\lambda_r\lambda_k}
	\frac 1{1+\lambda_r\lambda_j}
	\lambda_r^2 a_{kj}^2.
	\end{align}
	When $\ell=q=r,$ the fourth term in the parenthesis is the only non-zero term and we get the term containing $a_{jk}^2$ when $\set{m,p}=\set{j,k}.$
	This then leads to
	\begin{align}\label{T2-2}
	\sum_{\substack{r\neq j,k\\k<r<j}}
	\frac 1{1+\lambda_r\lambda_j}
	\frac 1{1+\lambda_k\lambda_r}
	\lambda_r^2 a_{jk}^2
	+ \sum_{\substack{r\neq j,k\\j<r<k}}
	\frac 1{1+\lambda_r\lambda_k}
	\frac 1{1+\lambda_j\lambda_r}
	\lambda_r^2 a_{kj}^2.
	\end{align}
	When $m=p=r,$ the third term in the parenthesis is the only non-zero term and we get the term containing $a_{jk}^2$ when $\set{\ell,q}=\set{j,k}.$
	This then leads to
	\begin{align}\label{T2-3}
	\sum_{\substack{r\neq j,k\\j<r<k}}
	\frac 1{1+\lambda_r\lambda_j}
	\frac 1{1+\lambda_r\lambda_k}
	\lambda_r^2 a_{jk}^2
	+ \sum_{\substack{r\neq j,k\\k<r<j}}
	\frac 1{1+\lambda_r\lambda_k}
	\frac 1{1+\lambda_r\lambda_j}
	\lambda_r^2 a_{kj}^2.
	\end{align}
	When $m=q=r,$ the third term in the parenthesis is the only non-zero term and we get the term containing $a_{jk}^2$ when $\set{\ell,p}=\set{j,k}.$
	This then leads to
	\begin{align}\label{T2-4}
	\sum_{\substack{r\neq j,k\\r>j,r>k}}
	\frac 1{1+\lambda_r\lambda_j}
	\frac 1{1+\lambda_r\lambda_k}
	\lambda_r^2 a_{jk}^2
	+ \sum_{\substack{r\neq j,k\\r>k,r>j}}
	\frac 1{1+\lambda_r\lambda_k}
	\frac 1{1+\lambda_r\lambda_j}
	\lambda_r^2 a_{kj}^2.
	\end{align}
	Combining these four terms \eqref{T2-1}, \eqref{T2-2}, \eqref{T2-3}, and \eqref{T2-4}, we get exactly two summations over $r\neq j,k,$ so the terms involving $a_{jk}^2$ in $T_2''$ are 
	\begin{align*}
	\sum_{r\neq j,k}
	\frac 1{1+\lambda_r\lambda_j}
	\frac 1{1+\lambda_r\lambda_k}
	\lambda_r a_{jk}^2
	+ \sum_{r\neq j,k}
	\frac 1{1+\lambda_r\lambda_k}
	\frac 1{1+\lambda_r\lambda_j}
	\lambda_r a_{kj}^2
	= \sum_{r\neq j,k}
	\frac{2\lambda_r^2}
	{(1+\lambda_j\lambda_r)(1+\lambda_k\lambda_r)}
	\, a_{jk}^2.
	\end{align*}
	This completes the proof of the lemma.
\end{proof}

\begin{lemma}[$T_3$'s expansion]
	\label{lem:T3}
	We have
	\begin{align*}
	T_3
	=2 \operatorname{tr}\big(M (S \wedge A S G A + A S G A \wedge S)\big)
    &=2 \sum_{j<k}
	\frac 1{1+\lambda_j\lambda_k}
	\pr{\sum_r a_{jr}^2
		\frac{\lambda_r \lambda_k} {1+\lambda_r^2} 
		+ \sum_r  a_{kr}^2
		\frac{\lambda_j\lambda_r}
		{1+\lambda_r^2}
	}\\
	&=T_3'+T_3'',
	\end{align*}
	where 
	\begin{align*} 
	T_3'&=  2 
	\sum_{j<k}\frac 1{1+\lambda_j\lambda_k}
	\pr{
		\frac{\lambda_j\lambda_k}{1+\lambda_j^2} a_{jj}^2
		+ \frac{\lambda_k^2}{1+\lambda_k^2} a_{jk}^2
		+ \frac{\lambda_j^2}{1+\lambda_j^2} a_{kj}^2
		+ \frac{\lambda_j\lambda_k}{1+\lambda_k^2} a_{kk}^2
	}\text{ and}\\
	T_3''&= \sum_{j<k}
	\sum_{m\neq j,k}
	\pr{
		\frac{2\lambda_j\lambda_m}{(1+\lambda_j^2)(1+\lambda_k\lambda_m)}
		+ \frac{2\lambda_k\lambda_m}{(1+\lambda_k^2)(1+\lambda_j\lambda_m)}
	}a_{jk}^2.
	\end{align*}
\end{lemma}

\begin{proof}
	As in Lemma~\ref{lem:T1}, we work with an orthonormal basis $\set{v_j}_{j=1,\cdots,n}$ which diagonalizes $S={\rm diag}(\lambda_1,\cdots,\lambda_n)$ and write $A=(a_{jk})_{j,k}$ in the same orthonormal basis.
	Again, $v_j\wedge v_k$'s form a basis of $\wedge^2\bb R^n$ for $j<k.$
	Thus, if $b_{jk}$ is the $(j,k)$-entry of $ASGA,$ i.e.,
	\begin{align}\label{bjk}
	b_{jk}
	= \sum_{r} a_{jr}\cdot \lambda_r\cdot \frac 1{1+\lambda_r^2}\cdot a_{rk},
	\end{align}
	then the $(jk,pq)$-entry of $ASGA \wedge S + S \wedge ASGA$ is
	\begin{align*}
	b_{jp}\lambda_k\delta_{kq} - b_{jq}\lambda_k\delta_{kp}
	+ \lambda_j\delta_{jp}b_{kq}
	- \lambda_j\delta_{jq}b_{kp},
	\end{align*}
	and hence
	\begin{align*}
	\tr\pr{M (ASGA \wedge S + S \wedge ASGA) }
	= & \sum_{j<k}
	m_{jk} 
	\pr{b_{jj}\lambda_k\delta_{kk} - b_{jk}\lambda_k\delta_{kj}
		+ \lambda_j\delta_{jj}b_{kk}
		- \lambda_j\delta_{jk}b_{kj}}\\
	= & \sum_{j<k} m_{jk} 
	\pr{b_{jj}\lambda_k + b_{kk}\lambda_j }.
	\end{align*}
	Combining this with \eqref{bjk} leads to the first expression of $T_3$ claimed in the lemma.
	
	Next, as in Lemma~\ref{lem:T2}, we split $T_3$ into two terms $T_3'$ and $T_3''$ by
	\begin{align*}
	T_3'
	&:=2 \sum_{j<k}
	\frac 1{1+\lambda_j\lambda_k}
	\pr{\sum_{r=j,k} a_{jr}^2
		\frac{\lambda_r \lambda_k} {1+\lambda_r^2} 
		+ \sum_{r=j,k} a_{kr}^2
		\frac{\lambda_j\lambda_r}
		{1+\lambda_r^2}
	}
	\text{ and}\\
	T_3''
	&:=2 \sum_{j<k}
	\frac 1{1+\lambda_j\lambda_k}
	\pr{\sum_{r\neq j,k} a_{jr}^2
		\frac{\lambda_r \lambda_k} {1+\lambda_r^2} 
		+ \sum_{r\neq j,k}  a_{kr}^2
		\frac{\lambda_j\lambda_r}
		{1+\lambda_r^2}
	}.
	\end{align*}
	Thus, there are exactly four terms in $T_3'$ and the expression of $T_3'$ claimed in the lemma follows directly.
	
	For $T_3'',$ we will do a similar analysis as in Lemma~\ref{lem:T2}.
	We first rewrite the expression of $T_3''$ as 
	\begin{align*}
	T_3''
	= 2 \sum_{\ell<m}
	\sum_{r\neq \ell,m}
	\pr{
		a_{\ell r}^2
		\frac{\lambda_r \lambda_m}{(1+\lambda_r^2)(1+\lambda_\ell\lambda_m)}
		+ a_{mr}^2
		\frac{\lambda_\ell\lambda_r}
		{(1+\lambda_r^2)(1+\lambda_\ell\lambda_m)}
	}.
	\end{align*}
	We fix two numbers $j<k.$
	
	We first look at the first term involving $a_{\ell r}^2.$
	We then get non-trivial contribution from the term when $\set{\ell,r}=\set{j,k}.$
	Thus, we get the term containing $a_{jk}^2$ from the first term to be
	\begin{align}\label{T3-1}
	2\sum_{\substack{m>j\\m\neq k}}
	a_{jk}^2
	\frac{\lambda_k \lambda_m}{(1+\lambda_k^2)(1+\lambda_j\lambda_m)}
	+ 2\sum_{\substack{m>k\\m\neq j}}
	a_{kj}^2
	\frac{\lambda_j \lambda_m}{(1+\lambda_j^2)(1+\lambda_k\lambda_m)}.
	\end{align}
	We next look at the first term involving $a_{m r}^2.$
	We then get non-trivial contribution from the term when $\set{m,r}=\set{j,k}.$
	Thus, we get the term containing $a_{jk}^2$ from the first term to be
	\begin{align}\label{T3-2}
	2\sum_{\substack{\ell<j\\\ell\neq k}}
	a_{jk}^2
	\frac{\lambda_\ell\lambda_k}
	{(1+\lambda_k^2)(1+\lambda_\ell\lambda_j)}
	+ 2\sum_{\substack{\ell<k\\\ell\neq j}}
	a_{kj}^2
	\frac{\lambda_\ell\lambda_j}
	{(1+\lambda_j^2)(1+\lambda_\ell\lambda_k)}.
	\end{align}
	
	If we replace the indices $m$'s and $\ell$'s in \eqref{T3-1} and \eqref{T3-2} with $r$'s, we get exactly two summations over $r\neq j,k$ and get the terms involving $a_{jk}^2$ in $T_3''$ to be 
	\begin{align*}
	2\sum_{r\neq j,k}
	a_{jk}^2
	\frac{\lambda_r\lambda_k}{(1+\lambda_k^2)(1+\lambda_r\lambda_j)}
	+ 2\sum_{r\neq j,k}
	a_{kj}^2
	\frac{\lambda_r\lambda_j}{(1+\lambda_j^2)(1+\lambda_r\lambda_k)}.
	\end{align*}
	The lemma then follows from $a_{jk}=a_{kj}.$
\end{proof}

Finally, we also rewrite $T_4$ in the same way.
It is a straightforward splitting but we still write it in a lemma.

\begin{lemma}[$T_4$'s expansion]
	\label{lem:T4}
	We have
	\[\begin{split}T_4
    =-(n-1) \operatorname{tr}(G A G A) -(n-1) \operatorname{tr}(G S A G S A)
    &= -(n-1)\sum_{j,k}
	\frac{a_{jk}^2\pr{1+\lambda_j\lambda_k}}{(1+\lambda_j^2)(1+\lambda_k^2)}\\
	&=T_4'+T_4'',
	\end{split}\]
	where 
	\begin{align*}
	T_4'&=- \sum_{j<k}\pr{
		\frac{1}{1+\lambda_j^2} a_{jj}^2
		+2 \frac{\pr{1+\lambda_j\lambda_k}}{(1+\lambda_j^2)(1+\lambda_k^2)} a_{jk}^2
		+ \frac{1}{1+\lambda_k^2} a_{kk}^2
	}\text{ and}\\
	T_4''&=- 2(n-2)\cdot \sum_{j<k} 
	\frac{a_{jk}^2\pr{1+\lambda_j\lambda_k}}{(1+\lambda_j^2)(1+\lambda_k^2)}.
	\end{align*}
	
\end{lemma}

\begin{proof}
	First, if we write $g_j:=1/(1+\lambda_j^2)$ to be the entry of $G,$ we can expand
	\begin{align*}
	\tr(GAGA) + \tr(GSAGSA)
	= \sum_{j,k=1}^n \pr{g_ja_{jk}g_ka_{kj} 
		+g_j\lambda_ja_{jk}g_k\lambda_k a_{kj}
	}
	& = \sum_{j,k} g_jg_ka_{jk}^2(1+\lambda_j\lambda_k)\\
	& = \sum_{j,k}
	\frac{a_{jk}^2\pr{1+\lambda_j\lambda_k}}{(1+\lambda_j^2)(1+\lambda_k^2)}.
	\end{align*}
	This proves the first expression of $T_4.$
	
	Next, we split $T_4$ into
	\begin{align*}
	T_4
	= -(n-1) \sum_{j,k}
	\frac{a_{jk}^2\pr{1+\lambda_j\lambda_k}}{(1+\lambda_j^2)(1+\lambda_k^2)}
	& = -(n-1) 
	\pr{\sum_{j=k}
		\frac{a_{jk}^2\pr{1+\lambda_j\lambda_k}}{(1+\lambda_j^2)(1+\lambda_k^2)}
		+ \sum_{j\neq k}
		\frac{a_{jk}^2\pr{1+\lambda_j\lambda_k}}{(1+\lambda_j^2)(1+\lambda_k^2)}}\\
	& = -(n-1)\sum_j \frac{a_{jj}^2}{1+\lambda_j^2}
	- 2(n-1) \sum_{j<k} \frac{a_{jk}^2\pr{1+\lambda_j\lambda_k}}{(1+\lambda_j^2)(1+\lambda_k^2)}.
	\end{align*}
	This implies $T_4'+T_4''=T_4.$
	Note that in $T_4',$ each $a_{jj}^2$ term appears $(n-1)$ times, matching $T_4$ above;
	each $a_{jk}^2$ appears $2$ times, so combined with $T_4''$ where each $a_{jk}^2$ appears $2(n-2)$ times, it matches the $2(n-1)$ times in $T_4.$
\end{proof}

We will combine the results in Lemmas~\ref{lem:T1}, \ref{lem:T2}, \ref{lem:T3}, and \ref{lem:T4} in the following two corollaries.
First, we collect the diagonal terms $T_k'$'s.
Recall that we have $T_1'=T_1.$

\begin{cor}
	\label{lem:diag-T}
	We have
	\begin{align*}
	T_1'+T'_2+T'_3+T'_4
	= - \sum_{j<k}
	\pr{\frac 1{\pr{1+\lambda_j\lambda_k}^2} 
		\pr{\sqrt{\frac{1+\lambda_k^2}{1+\lambda_j^2}}a_{jj} - \sqrt{\frac{1+\lambda_j^2}{1+\lambda_k^2}}a_{kk}}^2
		+\frac {4}{(1+\lambda_j^2)(1+\lambda_k^2)}a_{jk}^2}.
	\end{align*}
\end{cor}

\begin{proof}
	From Lemmas~\ref{lem:T1}, \ref{lem:T2}, \ref{lem:T3}, and \ref{lem:T4}, we can write
	$T_1'+T_2'+T_3'+T_4' =\sum\limits_{j<k} \Xi_{jk}'$
	with
	\begin{align*}
	\Xi'_{jk}
	= & 
	2 
	\frac 1{1+\lambda_j\lambda_k} 
	\pr{a_{jj}a_{kk} - a_{jk}^2}\\
	& - 
	\frac 1{(1+\lambda_j\lambda_k)^2} 
	\pr{\lambda_k a_{jj}
		+ \lambda_j a_{kk}
	}^2\\
	& + 2 
	\frac 1{1+\lambda_j\lambda_k}
	\pr{
		\frac{\lambda_j\lambda_k}{1+\lambda_j^2} a_{jj}^2
		+ \frac{\lambda_k^2}{1+\lambda_k^2} a_{jk}^2
		+ \frac{\lambda_j^2}{1+\lambda_j^2} a_{kj}^2
		+ \frac{\lambda_j\lambda_k}{1+\lambda_k^2} a_{kk}^2
	}\\
	& - \pr{
		\frac{1}{1+\lambda_j^2} a_{jj}^2
		+2 \frac{\pr{1+\lambda_j\lambda_k}}{(1+\lambda_j^2)(1+\lambda_k^2)} a_{jk}^2
		+ \frac{1}{1+\lambda_k^2} a_{kk}^2
	}.
	\end{align*}
    By looking at the coefficients of $a_{jj}^2,a_{jj}a_{kk},$ and $a_{jk}^2,$ a straightforward calculation leads to 
	\begin{align*}
	\Xi_{jk}'
	& = \frac{-1-\lambda_k^2}{(1+\lambda_j^2)(1+\lambda_j\lambda_k)^2} a_{jj}^2
	+ \frac 2{(1+\lambda_j\lambda_k)^2} a_{jj}a_{kk}
	+ \frac{-1-\lambda_j^2}{(1+\lambda_k^2)(1+\lambda_j\lambda_k)^2} a_{kk}^2
	-\frac{4}
	{(1+\lambda_j^2)(1+\lambda_k^2)} a_{jk}^2.
	\end{align*}
    The lemma then follows by completing the square.
\end{proof}

Next, we collect the off-diagonal terms $T_k''$'s.
Recall that $T_1''=0.$

\begin{cor}
	\label{lem:off-diag-T}
	We have
	\begin{align*}
	T_2''+T_3''+T_4''
	= -\sum_{j<k}
	\sum_{m\neq j,k}
	\frac{2(1+\lambda_j\lambda_k)(1+\lambda_m^2)}
	{(1+\lambda_j^2)(1+\lambda_k^2)(1+\lambda_j\lambda_m)(1+\lambda_k\lambda_m)}
	a_{jk}^2.
	\end{align*}
\end{cor}

\begin{proof}
	From Lemmas~\ref{lem:T1}, \ref{lem:T2}, \ref{lem:T3}, and \ref{lem:T4}, we can write
	$T_2''+T_3''+T_4''
	=\sum\limits_{j<k} \Xi_{jk}''$
	with
	\begin{align*}
	\Xi''_{jk}
	= &
	-\sum_{j<k}\sum_{\substack{m\neq j,k}}
	\frac{2\lambda_m^2}
	{(1+\lambda_j\lambda_m)(1+\lambda_k\lambda_m)}
	\, a_{jk}^2\\
	& +  \sum_{j<k}
	\sum_{m\neq j,k}
	\pr{
		\frac{2\lambda_j\lambda_m}{(1+\lambda_j^2)(1+\lambda_k\lambda_m)}
		+ \frac{2\lambda_k\lambda_m}{(1+\lambda_k^2)(1+\lambda_j\lambda_m)}
	}a_{jk}^2\\
	& - 2(n-2)\cdot \sum_{j<k} 
	\frac{a_{jk}^2\pr{1+\lambda_j\lambda_k}}{(1+\lambda_j^2)(1+\lambda_k^2)}.
	\end{align*}
	Therefore, $\Xi_{jk}''$ is the sum of
	\begin{align*}
	\pr{
		-\frac{2\lambda_m^2}
		{(1+\lambda_j\lambda_m)(1+\lambda_k\lambda_m)}
		+ \frac{2\lambda_j\lambda_m}{(1+\lambda_j^2)(1+\lambda_k\lambda_m)}
		+ \frac{2\lambda_k\lambda_m}{(1+\lambda_k^2)(1+\lambda_j\lambda_m)}
		- \frac{2(1+\lambda_j\lambda_k)}{(1+\lambda_j^2)(1+\lambda_k^2)}
	}a_{jk}^2
	\end{align*}
	over $j<k$ and $m\neq j,k.$
    Via a straightforward calculation, it follows that
    \begin{align*}
    &-\frac{2\lambda_m^2}
		{(1+\lambda_j\lambda_m)(1+\lambda_k\lambda_m)}
		+ \frac{2\lambda_j\lambda_m}{(1+\lambda_j^2)(1+\lambda_k\lambda_m)}
		+ \frac{2\lambda_k\lambda_m}{(1+\lambda_k^2)(1+\lambda_j\lambda_m)}
		- \frac{2(1+\lambda_j\lambda_k)}{(1+\lambda_j^2)(1+\lambda_k^2)}\\
    = &-\frac{2(1+\lambda_j\lambda_k)(1+\lambda_m^2)}
	{(1+\lambda_j^2)(1+\lambda_k^2)(1+\lambda_j\lambda_m)(1+\lambda_k\lambda_m)}
    \end{align*}
    and this finishes the proof. 
\end{proof}

We remark that all the terms $T''_k$'s vanish when $n=2.$
In fact, this is essentially how we split the terms $T_k$'s.
Corollary~\ref{lem:diag-T} basically deals with the two-dimensional case, while Corollary~\ref{lem:off-diag-T} deals with the additional terms arising in the high dimensional case.

We are in a position to prove Proposition~\ref{prop:Hess-n-dim}.

\begin{proof}
	[Proof of Proposition~\ref{prop:Hess-n-dim}]
	We write the Hessian of $\log \Lambda(S)$ in the direction $(A,A)$ as the sum of $T_1,T_2,T_3,$ and $T_4$ given in \eqref{T-1-T-4}.
	Combining Corollaries~\ref{lem:diag-T} and~\ref{lem:off-diag-T} proves Proposition~\ref{prop:Hess-n-dim}.
\end{proof}

\section{\bf Bernstein theorems for almost isoclinic minimal Lagrangian submanifolds}
\label{sec:Bernstein}

We will use Theorem~\ref{thm:Lambda-elliptic-equation} to prove Bernstein-type theorems for minimal Lagrangian cones and minimal Lagrangian submanifolds.

\subsection{Almost isoclinic minimal cones}

In this section, we will work on minimal Lagrangian cones.
The main ingredient is Theorem~\ref{thm:Lambda-elliptic-equation}. 



\begin{lemma}
	\label{lem:cone-rigidity}
	Let $\mathcal C$ be a regular minimal Lagrangian cone in $\bb R^{2n}.$
	If $\mathcal C$ is almost isoclinic on its smooth part, then $\mathcal C$ is the union of some Lagrangian $n$-planes.
\end{lemma}

\begin{proof}
	On ${\rm reg}\,\mathcal C,$ the smooth part of $\mathcal C,$ the quantity $\Lambda$ is positive and satisfies an inequality given by Theorem~\ref{thm:Lambda-elliptic-equation}.
	We look at its link $S:=\mathcal C\cap \mathbb S^{2n-1}$ and the quantity $\Lambda_S:=\Lambda|_S.$
	
	First, we note that $\Lambda$ is a scale-invariant quantity.
	That is, given $x\in {\rm reg}\,\mathcal C,$ we have
	\begin{align}\label{Lambda-scale-invariant}
	\Lambda(x) = \Lambda_S\pr{\frac{x}{|x|}}.
	\end{align}
	In particular, since $\Lambda_S>0$ has a positive lower bound on $S,$ we know that there exists $c>0$ such that $\Lambda\ge c$ on ${\rm reg}\,\mathcal C.$
	
	Next, if we let $\D_{\mathcal C}$ and $\D_S$ be the Laplace operators on ${\rm reg}\,\mathcal C$ and $S,$ then they are related by
	\begin{align*}
	\D_{\mathcal C} \log \Lambda
	= \bd_r^2 \log \Lambda
	+ \frac{n-1}r \bd_r \log \Lambda
	+ \frac 1{r^2}\D_S \log \Lambda_S.
	\end{align*}
	By the scale-invariant property~\eqref{Lambda-scale-invariant}, we can use Theorem~\ref{thm:Lambda-elliptic-equation} to derive that on $S,$
	\begin{align*}
	\D_S\log \Lambda_S
	= r^2 \D_{\mathcal C}\log \Lambda
	\le -c_n|A_{\mathcal C}|^2
	\end{align*}
	for some $c_n>0.$
	Integrating the inequality by parts on $S$ then implies $A_{\mathcal C}=0$ on $S.$
	This implies that $S$ is the union of $(n-1)$-dimensional round spheres in $\mathbb S^{2n-1},$ and the conclusion follows.	
\end{proof}

\subsection{Almost isoclinic minimal Lagrangians}

We will use Lemma~\ref{lem:cone-rigidity} to prove Theorem~\ref{thm:Bernstein-AC}.
A key step is to notice that a positive lower bound of $\Lambda$ and the minimality condition actually imply the Lagrangian is graphical after a rotation.

\begin{prop}
\label{prop:graphicality-of-AI-minimal}
    Let $M^n$ be a complete connected minimal Lagrangian in $\bb R^{2n}.$
	If there exists $c>0$ such that $\Lambda\ge c$ on $M,$ then $M$ is graphical up to a rotation.
    To be precise, there exist $\varphi\in\bb R$ and a (vector-valued) smooth function $V\colon \bb R^n\to \bb R^n$ with $|\n V|\le C<\infty$ such that
	\begin{align*}
	e^{i\varphi} M = {\rm Graph}_{\bb R^n} V.
	\end{align*}
\end{prop}

\begin{proof}
    Let $\Theta$ be the Lagrangian angle of $M,$ which is a constant since $M$ is connected.
    For any $p\in M,$ there exist $\varphi\in\bb R$ and an $n$-plane $L_\varphi=e^{i\varphi}L_B$ so that $T_pM$ is graphical over $L_\varphi$ and satisfies the condition in Definition~\ref{def:AI}.
    Write $T_pM$ as the graph of a symmetric matrix $S_\varphi$ over $L_0,$ and let $\lambda_k$'s be the eigenvalues of $S_\varphi.$
    Thus, if we let $\theta_k:=\arctan \lambda_k$ be the characteristic angles in this chart, it follows that
    \begin{align}\label{SL-ang}
    e^{i\Theta} = e^{i\pr{n\varphi + \sum_k\theta_k}}.
    \end{align}
    Consider
    \begin{align*}
    \psi:=\varphi + \frac 1n \sum_k\theta_k,
    \end{align*}
    and hence \eqref{SL-ang} becomes $e^{i\Theta} = e^{in\psi}.$
    Thus, $n\psi$ is independent of $p\in M$ up to a multiple of $2\pi.$
    Moreover, if we look at the new reference plane $L_\psi=e^{i\psi}L_B,$ according to Lemma~\ref{lem:lambda-rotate}, $T_pM$ is graphical over $L_\psi$ and the eigenvalues of the graph matrix are
    \begin{align*}
    \td \lambda_k 
    = \frac{\lambda_k\cos\ovl\theta - \sin\ovl\theta}{\cos\ovl\theta + \lambda_k\sin\ovl\theta}
    = \tan(\theta_k-\ovl\theta)
    \end{align*}
    where we write $\ovl\theta:=\frac 1n \sum_k\theta_k.$
    Note that each $\theta_k-\ovl\theta$ still lies in $(-\pi/2,\pi/2),$ so the new characteristic angles in this chart are
    $$\td\theta_k:=\arctan\td\lambda_k=\theta_k-\ovl\theta.$$
    Therefore, we find an open cover of $M$ by
    \begin{align*}
    M\sbst \bigcup_{t=0}^{n-1}
    \mathcal O^M_{\frac{\Theta+2\pi t}n}
    \end{align*}
    where we let $\mathcal O^M_\alpha
    := \set{q\in M:T_qM\in \mathcal O_\alpha}.$

    We claim that each piece $\mathcal O^M_{\frac{\Theta+2\pi t}n}$ is a closed subset in $M.$
    Fix a $t\in\set{0,1,\cdots,n-1}$ and write $\beta:=(\Theta+2\pi t)/n.$
    Suppose $p\in \mathcal O^M_\beta$ and let $\td\theta_k$'s be the characteristic angles in this chart.
    These angles satisfy $\sum_k\td\theta=0$ and
    \begin{align*}
    \prod_{j<k}\cos\pr{\td\theta_j-\td\theta_k}
    =\Lambda\ge c.
    \end{align*}
    This implies each $\cos\pr{\td\theta_j-\td\theta_k}$ is bounded by $c$ from below.
    Thus, 
    \begin{align*}
    |\td\theta_j-\td\theta_k|\le \arccos c\in \left[0, \frac\pi 2\right)
    \end{align*}
    for all $j$ and $k.$
    This and the condition $\sum_k\td\theta=0$ imply
    \begin{align}\label{theta-k-bdd}
    |\td\theta_k|\le \frac{n-1}n\arccos c.
    \end{align}
    This implies that if $p_\ell$ is a sequence of points in $\mathcal O^M_{\frac{\Theta+2\pi t}n}$ that converges to some $p_\infty\in M,$ then the characteristic angles of $T_{p_\infty}M$ still satisfy the same estimates \eqref{theta-k-bdd}, so $p_\infty\in O^M_{\frac{\Theta+2\pi t}n}.$
    This proves that each $\mathcal O^M_{\frac{\Theta+2\pi t}n}$ is closed.
    The connectivity of $M$ then implies 
    \begin{align*}
    M = \mathcal O^M_{\frac{\Theta+2\pi t_0}n}
    \end{align*}
    for a common $t_0\in\set{0,1,\cdots,n-1}.$ 
    This and the gradient bound \eqref{theta-k-bdd} allow us to conclude that the orthogonal projection 
    \begin{align*}
    \pi\colon \bb R^{2n}\to L_M:= e^{i\frac{\Theta+2\pi t_0}n} L_B
    \end{align*}
    restricts to a local diffeomorphism on $M.$

    We will finish the proof by showing that $\pi|_M$ is a covering map.
    We claim that every $C^1$ path in $L_M$ admits a lift to $M$ from any prescribed point in the appropriate fiber. 
    Let $\gamma\colon[0,1]\to L_M$ be a $C^1$ path, and let $p\in M$ satisfy $\pi(p)=\gamma(0).$ 
    Since $\pi$ is a local diffeomorphism, there is a unique local lift
    $\td\gamma(t)$ with $\td\gamma(0)=p.$ 
    Let $[0,T)$ be its maximal interval of existence.
    Along the lift, the gradient estimate \eqref{theta-k-bdd} gives
    $$|\td\gamma'(t)|\le \pr{\cos\pr{\frac{n-1}n\arccos c}}^{-1} |\gamma'(t)|.$$
    Note that $\arccos c<\pi/ 2.$
    Consequently,
    \begin{align*}
    |\td\gamma(t)-\td\gamma(0)|
    &\le \pr{\cos\pr{\frac{n-1}n\arccos c}}^{-1} \cdot \int_0^t|\gamma'(s)|ds\\
    &\le \pr{\cos\pr{\frac{n-1}n\arccos c}}^{-1} \cdot \operatorname{Length}(\gamma).
    \end{align*}
    Thus, if $T<1,$ $\td\gamma$ would have finite length and the completeness of $M$ would allow us to extend $\td\gamma$ up to time $t=T.$
    Because $\pi$ is a local diffeomorphism near $\td\gamma(T),$ the lift extends past $T,$ contradicting maximality. 
    Thus, $T=1.$
    It follows that $\pi|_M$ has the unique path-lifting property, and hence it is a covering map.
    Since $L_M\simeq \bb R^n$ is simply connected, the conclusion above implies that $\pi|_M$ is a global diffeomorphism, and hence $M$ is given by the graph of some $V\colon L_M\to\bb R^n.$
    The gradient of~$V$ is bounded by the gradient bound \eqref{theta-k-bdd} again.
\end{proof}

Combining Proposition~\ref{prop:graphicality-of-AI-minimal} with Lemma~\ref{lem:cone-rigidity} implies the main Bernstein theorem.

\begin{proof}
	[Proof of Theorem~\ref{thm:Bernstein-AC}]
	By Proposition~\ref{prop:graphicality-of-AI-minimal}, $M$ is the graph of a function with bounded gradient.
    In particular, $M$ is asymptotically conical with exactly one end so its blowdonw is a minimal Lagrangian cone with $\Lambda\ge c.$
	By Lemma~\ref{lem:cone-rigidity}, the cone is a Lagrangian $n$-plane.
	By the monotonicity formula for minimal submanifolds, this implies that $M$ itself is a Lagrangian $n$-plane.
\end{proof}

A Bernstein-type result of this kind is known to have an equivalent form, which we record in the following (cf. \cite{W16}*{Lecture 3}).

\begin{cor}
\label{cor:local-curv-est}
	Given $c>0,$ there exists $C=C(n,c)<\infty$ such that the following holds.
	If $M^n$ is a minimal Lagrangian submanifold with $\Lambda\ge c,$ then 
	\begin{align*}
	|A_M|\le \frac C{d_M(\cdot,\bd M)}.
	\end{align*} 
\end{cor}

\subsection{Almost isoclinic surfaces}
\label{sec:surface-Bernstein}

In this section, we prove a stronger Bernstein-type result for two-dimensional minimal Lagrangians that are almost isoclinic.
In place of the new estimate for $\log\Lambda,$ our main tools come from complex geometry and analysis, which allows us to work with a more general class of surfaces beyond almost isoclinic ones.

\begin{thm}
	\label{thm:Bernstein}
	Let $M$ be a complete, properly embedded Lagrangian minimal surface in $\bb R^{4}.$
	If $\Lambda\ge c$ for some $c>-1$, then $M$ is the union of Lagrangian $2$-planes.
\end{thm}

\begin{remark}
\label{rmk:sharpness}
    The same argument in the proof of Theorem~\ref{thm:Bernstein} below implies the same rigidity for a minimal Lagrangian with $\Lambda\le c<1.$
    Explicit calculations as in Section~\ref{sec:Lawlor-neck} imply that on a Lawlor neck in $\bb C^2$, the supremum and infimum of $\Lambda$ are $1$ and~$-1.$
    Thus, the bound in Theorem~\ref{thm:Bernstein} is sharp.
    In terms of the language of Gauss maps (see Figure~\ref{fig:Gauss}), the theorem says that if the projection to the anti-self-dual $\mathbb S^2$ avoids a neighborhood of the north or south pole, then the minimal Lagrangian is flat.
    From the perspective of graphs of symplectomorphisms (or area-preserving maps) as discussed in Section~\ref{sec:2D-AI}, Theorem~\ref{thm:Bernstein} should be compared with the results in \cite{Ni}.
\end{remark}

In dimension two, the estimate in Theorem~\ref{thm:Lambda-elliptic-equation} is stronger and can be used to obtain a PDE proof for Theorem~\ref{thm:Bernstein} for a properly immersed minimal Lagrangian with a positive lower bound on~$\Lambda.$

\begin{proof}
[Proof of Theorem~\ref{thm:Bernstein}]
    Recall that $M$ is the image of a proper embedding into $\mathbb{R}^4$.
    We may assume that $M$ is connected and special Lagrangian with respect to the complex coordinates $z_j=x_j+iy_j, j=1, 2$.
    Consider the new coordinates $w_1=x_1+ix_2$ and $w_2=y_1-iy_2,$
    which induce a new complex structure.
    Then $M$ is a smooth holomorphic curve with respect to this new complex structure.
    Since $M$ is properly embedded, in particular, $M$ is topologically closed. 
    Since every line bundle on $\bb C^2$ is holomorphically  trivial, such a closed holomorphic curve is the zero set of a holomorphic function $h\colon \bb C^2\to\bb C.$
    
    Let
    $\omega = \frac i2 \pr{dw_1\wedge d\ovl w_1 + dw_2\wedge d\ovl w_2}$
    be the new K\"ahler form and let 
    \begin{align*}
    \beta
    :& = \frac i2 \pr{dw_1\wedge d\ovl w_1 - dw_2\wedge d\ovl w_2}
    = dx_1\wedge dx_2 + dy_1\wedge dy_2,
    \end{align*}
    so $\Lambda = *\beta|_L.$
    At a point on $L,$ the tangent space is spanned by $\xi=(\bd_{w_1}h, \bd_{w_2}h),$ so
    \begin{align*}
    \Lambda
    = *\beta|_M
    = \frac{\beta\pr{\xi, J\xi}}{\omega(\xi,J\xi)}
    = \frac{|\bd_{w_2}h|^2-|\bd_{w_1}h|^2} {|\bd_{w_2}h|^2+|\bd_{w_1}h|^2}.
    \end{align*}
    Thus, the bound $\Lambda\ge c$ implies
    \begin{align}\label{L-w-2-dom}
    |\bd_{w_2}h|^2 \ge \frac{1+c}{1-c}|\bd_{w_1}h|^2
    \end{align} 
    on $M.$ 
    Since $c>-1,$ this and the smoothness of $M$ implies $|\bd_{w_2}h|>0.$
    In particular, the implicit function theorem implies that the projection 
    $$\pi\colon (w_1,w_2)\in\bb C^2\mapsto w_1\in\bb C$$ 
    restricted to $M$ is a local diffeomorphism.

    Next, we observe that the properness of the embedding $M\to\bb C^2$ and the bound \eqref{L-w-2-dom} imply that~$\pi|_M$ is a covering map.
    It is similar to the second part of the proof of Proposition~\ref{prop:graphicality-of-AI-minimal}.
    We claim that every $C^1$ path in $\bb C$ admits a lift to $M$ from any prescribed point in the appropriate fiber. 
    Let $\gamma\colon[0,1]\to\bb C$ be a $C^1$ path, and let $p\in M$ satisfy $\pi(p)=\gamma(0).$ 
    Since $\pi$ is a local diffeomorphism, there is a unique local lift
    $$\td\gamma(t)=\pr{\gamma(t),\eta(t)}$$
    with $\td\gamma(0)=p.$ 
    Let $[0,T)$ be its maximal interval of existence.
    Along the lift, the estimate \eqref{L-w-2-dom} gives
    $|\eta'(t)|\le \frac{1-c}{1+c} |\gamma'(t)|.$
    Consequently,
    $$|\eta(t)-\eta(0)|
    \le \frac{1-c}{1+c} \cdot \int_0^t|\gamma'(s)|ds
    \le \frac{1-c}{1+c} \cdot \operatorname{Length}(\gamma).$$
    Thus, if $T<1,$ both coordinates of $\td\gamma(t)$ would remain in a fixed compact subset of $\mathbb C$ as $t\nwarrow T.$ 
    Hence, the image of $\td\gamma$ is contained in a compact subset $K\subset\bb C^2$.
    Since the embedding is proper, $M\cap K$
    is compact, so for any sequence $t_k\nearrow T$, a subsequence of $\td\gamma(t_k)$ converges to some point $q\in M.$ 
    By continuity, $\pi(q)=\gamma(T).$
    Because $\pi$ is a local diffeomorphism near $q,$ the lift extends past $T,$ contradicting maximality. 
    Thus, $T=1.$
    It follows that $\pi|_M$ has the unique path-lifting property, and hence it is a covering map.
    
    Since $\bb C$ is simply connected, the conclusion above implies that $\pi|_M$ is a global diffeomorphism, and hence $M$ is given by the graph of a holomorphic function $f\colon \bb C\to\bb C$ with $h\pr{w,f(w)}=0.$
    Thus, by \eqref{L-w-2-dom},
    \begin{align*}
    |\bd_w f| 
    = \abs{\frac{\bd_{w_1}h}{\bd_{w_2}h}}
    \le \frac{1-c}{1+c},
    \end{align*}
    so $f$ is a linear polynomial by the Liouville theorem.
    This implies that $L$ is a two-dimensional plane.
\end{proof}

\end{document}